\documentclass[12pt,a4paper]{article}
\usepackage{amsmath,amssymb,amsthm,graphicx,color}
\usepackage[left=1in,right=1in,top=1in,bottom=1in]{geometry}
\usepackage{cite}
\usepackage[british]{babel}
\usepackage{enumerate}
\usepackage{hyperref} 
\usepackage{bm}
\usepackage{comment}

\usepackage{tikz}
\usetikzlibrary{calc,arrows}

\hypersetup{
    colorlinks=false,
    pdfborder={0 0 0},
}

\newtheorem{theorem}{Theorem}[section]
\newtheorem{corollary}[theorem]{Corollary}

\newtheorem{lemma}[theorem]{Lemma}
\newtheorem{problem}[theorem]{Problem}

\numberwithin{equation}{section}

\theoremstyle{definition}

\theoremstyle{remark}
\newtheorem{remark}[theorem]{Remark}
\newtheorem{remarks}[theorem]{Remarks}
\newtheorem*{remark*}{Remark}

 \allowdisplaybreaks

\newcommand{\1}[1]{{\mathbf 1}{\{#1\}}}

\newcommand{\R}{{\mathbb R}}

\newcommand{\N}{{\mathbb N}}
\newcommand{\ZP}{{\mathbb Z}_+}
\newcommand{\RP}{{{\mathbb R}_+}}

\newcommand{\X}{{\mathbb X}}

\DeclareMathOperator{\Exp}{\mathbb{E}}
\renewcommand{\Pr}{{\mathbb P}}

\DeclareMathOperator{\cosec}{cosec} 
\DeclareMathOperator{\sign}{sgn}

\newcommand{\T}{\text{T}_{\alpha,\beta,c}}
\newcommand{\Tout}{\text{T}_{\alpha,\beta,c}^{\mathrm{out}}}
\newcommand{\Tin}{\text{T}_{\alpha,\beta,c}^{\mathrm{in}}}
\newcommand{\Tbal}{\text{T}_{\alpha,c}^{\mathrm{bal}}}

\newcommand{\eps}{\varepsilon}

\newcommand{\ud}{{\mathrm d}}

\newcommand{\cF}{{\mathcal F}}

\newcommand{\cR}{{\mathcal R}}

\newcommand{\cT}{{\mathcal T}}

\newcommand{\as}{\ \text{a.s.}}

\newcommand{\Bigmid}{\; \Bigl| \;}

\newcommand{\bx}{{\mathbf{x}}}

\newcommand{\bu}{{\mathbf{u}}}
\newcommand{\bv}{{\mathbf{v}}}

\newcommand{\0}{{\mathbf{0}}}

\makeatletter
\def\namedlabel#1#2{\begingroup  
    (#2)%
    \def\@currentlabel{#2}%
    \phantomsection\label{#1}\endgroup
}
\makeatother

\begin{document}

\title{Markov chains with heavy-tailed increments and asymptotically zero drift}
\author{Nicholas Georgiou\footnote{Department of Mathematical Sciences, Durham University, South Road, Durham DH1 3LE, UK}
\and Mikhail V.\ Menshikov\footnotemark[1]
\and Dimitri Petritis\footnote{IRMAR, Campus de Beaulieu, 35042 Rennes Cedex, France}
\and Andrew R.\ Wade\footnotemark[1]
}
\date{\today}
\maketitle

\begin{abstract}
We study the recurrence/transience phase transition for
Markov chains on $\RP$, $\R$, and $\R^2$ whose increments
have heavy tails with exponent in $(1,2)$  
and asymptotically zero mean. This is the infinite-variance analogue of the classical
Lamperti problem. On $\RP$, for example, we show that if the tail of the positive
increments is about $c y^{-\alpha}$ for an exponent $\alpha \in (1,2)$ and if the drift at $x$ is about $b x^{-\gamma}$,
then the critical regime has $\gamma = \alpha -1$ and recurrence/transience is determined by the 
sign of $b + c\pi \cosec (\pi \alpha)$. On $\R$ we classify whether transience is directional
or oscillatory, and extend an example of Rogozin \& Foss to a class of transient martingales which oscillate
between $\pm \infty$. In addition to our recurrence/transience results, we also give sharp
results on the existence/non-existence of moments of passage times.
\end{abstract}

\medskip


\medskip


\section{Introduction}
\label{sec:intro}

\emph{Lamperti's problem} describes how the asymptotic behaviour of a 
non-homogeneous random walk (Markov chain) $\xi_n$
on $\RP$ whose increments have at least 2 moments is determined by
the interplay between the increment moment functions
\begin{equation}
\label{eq:mu}
\mu (x) := \Exp [ \xi_{n+1} - \xi_n \mid \xi_n = x] , \text{ and } s^2 (x) :=  \Exp [ ( \xi_{n+1} - \xi_n)^2 \mid \xi_n = x] .\end{equation}
Lamperti~\cite{lamp1,lamp2,lamp3}
showed, among other things, that the critical regime for the recurrence or transience of $\xi_n$
is when $ x \mu (x)$ is comparable to $s^2(x)$.
For further background and results, see~\cite{dkw,mpow} and the references therein;
the continued interest of Lamperti's problem is due in part
to its nature as a prototypical near-critical stochastic system.

In this paper, we study the \emph{heavy-tailed} case where $s^2(x) = \infty$,
but $\mu(x)$ is still well-defined.
Under mild conditions, if $\limsup_{x \to \infty} \mu(x) < 0$
 then the Markov chain on~$\RP$ is positive recurrent,
while if $\liminf_{x \to \infty} \mu(x) > 0$
then the chain is transient (see e.g.~Theorems~2.6.2 and~2.5.18 of~\cite{mpow}). Thus the case of central interest
is the \emph{asymptotically zero drift} regime where $\mu(x)\to 0$
as $x \to \infty$. 

We show that in the heavy-tailed case it is the interplay between
the drift and the \emph{tails} of the increments that determines the
asymptotic behaviour.
For instance, suppose  
\begin{equation}
\label{eq:tail}
  \Pr[\xi_{n+1} -\xi_n > y \mid \xi_n = x] \sim c y^{-\alpha}
\end{equation}
for some constants $c>0$ and $\alpha \in (1,2)$, uniformly in $x$.
Here $s^2(x) = \infty$, so Lamperti's classification does not
apply.  We show (Theorem~\ref{thm:recurrence-on-RP} below) that if $\mu (x) \sim b x^{-\gamma}$
for some $\gamma \in (0,1)$, then the critical regime for
recurrence and transience has $\gamma = \alpha -1$,
and in that case the process is transient if $b + c \pi \cosec (\pi \alpha ) >0$
and recurrent if $b + c \pi \cosec (\pi \alpha ) < 0$.

As well as classifying recurrence and transience,
we quantify the recurrent cases by determining
which moments of return times to a bounded set
are finite.

In addition to considering chains on $\RP$,
we also consider chains on $\R$, and give an example on $\R^2$.
On $\R$, the situation is richer than on $\RP$,
and also than in the case where $s^2(x) = O(1)$,
since \emph{oscillatory} transience can occur,
when $\liminf_{n \to \infty} \xi_n = -\infty$ and $\limsup_{n \to \infty} \xi_n = +\infty$,
but nevertheless $\lim_{n \to \infty} | \xi_n | = \infty$.
In the case of \emph{zero drift} ($\mu(x) = 0$),
on $\RP$ the chain is recurrent under mild conditions (see Example~2.5.6 of~\cite{mpow})
but on $\R$ zero drift does \emph{not} imply recurrence when $s^2 (x) = \infty$
(cf.~Theorem~2.5.7 of~\cite{mpow}). Rogozin \& Foss~\cite{rf} gave a concrete example
on $\R$ of a zero-drift chain that is transient.
The example has heavy-tailed jumps, with exponent $\alpha \in (1,3/2]$,
 \emph{inwards}
to the origin from both the left and right half-lines. We show (Theorem~\ref{thm:recurrence-on-R-in} below)
that the transience in this case is oscillatory and
give general conditions
for behaviour of this kind, and also show how an asymptotically zero drift perturbs the picture.
Again, in the recurrent cases we also study moments of passage times.

Apart from being interesting in their own right, 
stochastic processes on the positive half-line are important in the study of
higher-dimensional processes via the Lyapunov function method (see e.g.~\cite{aim,fmm,lamp1,mpow}): for a Markov chain $\zeta_n$
on $\R^d$, the analysis typically proceeds by considering a
one-dimensional projection $\rho: \R^d \to \RP$ of the form $\rho(x) := \|x\|^{\nu}$, for some
positive constant $\nu$; then $\xi_n = \rho(\zeta_n)$ is a (not necessarily Markov) process on
$\RP$ and recurrence/transience of $\zeta_n$ and $\xi_n$ coincide. 
As a first step into higher dimensions, 
we give an example of a heavy-tailed random walk
on $\R^2$ and show how our approach can be used to provide conditions for recurrence and
transience.

We briefly mention other relevant literature.
The case of \emph{balanced} tails (when the left and right jumps
have the same exponent)
has been studied in several papers by Sandri\'c in both discrete and continuum settings: see~\cite{sandricB,sandricS,sandricT}
and references therein. The only previous results on passage time moments are due to Doney~\cite{doney}
in the case of a zero-drift, spatially homogeneous random walk, and a subset of the authors~\cite{mpew};
these include some special cases of our results, but not the asymptotically zero drift regime.
The case where even $\mu (x)$ is not well-defined is quite different: see~\cite{hmmw,mpew} and references therein.

Section~\ref{sec:results} presents our main results, working in turn through the cases
on $\RP$ and $\R$, and finally giving an example on $\R^2$ and an open problem for a reflected
random walk in a quadrant with heavy tails.
Section~\ref{sec:criteria} contains the tools that we need to
determine the asymptotic behaviour of our processes via the method of Lyapunov functions.
These include criteria for directional and oscillatory transience.
Section~\ref{sec:lyapunov} carries out the analysis of our Lyapunov functions.
Section~\ref{sec:proofs} contains the proofs of the main results.
Technical results on integral computations needed for our Lyapunov-function
estimates are collected in the Appendix.

\section{Main results}
\label{sec:results}

\subsection{Notation}
\label{sec:notation}

For all the models in this paper, we use the following notation.
Let $\X$ be an 
 infinite, unbounded, measurable
subset of  $\R^d$ ($d \in \N$), equipped with its own 
 $\sigma$-algebra $\mathcal{X}$.  To avoid unnecessary complications with conditioning, we suppose in the sequel that the measurable space $(\X, \mathcal{X})$  is a standard Borel space.
Let $\Xi := (\xi_n, n \in \ZP)$ be a time-homogeneous $\X$-valued Markov process. 
Here and elsewhere,
$\N := \{1,2,3,\ldots\}$ and $\ZP:= \N \cup \{ 0 \}$. 
The Markov process  $\Xi$ is determined by its Markov kernel $P:\X\times \mathcal{X}\to [0,1]$, specifying, as usual,  the conditional law $P(x,A)=\Pr(\xi_{n+1}\in A \mid \xi_n=x)$, for $x\in\X$ and $A\in\mathcal{X}$.  By the Ionescu Tulcea theorem, the kernel $P$ and the law of $\xi_0$ uniquely determine a probability measure $\Pr$ on the space $\X^{\ZP}$ that can be disintegrated into the conditional measures $\Pr_x[\, \cdot\, ]=\Pr[\; \cdot \mid \xi_0 = x]$ for all deterministic initial conditions. We write $\Exp_x$ for the
  expectation corresponding to $\Pr_x$. 
For $n \in \ZP$, 
define the increment $\theta_{n+1} :=
\xi_{n+1}-\xi_n$.
Consistently, the transition kernel of the chain is recovered  from the law of $\theta_1$ given $\xi_0$;
to ease notation, we write $\theta$ for $\theta_1$.

For $x \in \R$ we write $x_+ := x \1{ x \geq 0}$
and $x_- := - x \1 { x < 0 }$.

\subsection{Walks on the half line}
\label{sec:RP-results}

We start with $\X \subseteq \RP$.
We say that $\Xi$ is \emph{transient} if $\lim_{n \to \infty} \xi_n = \infty$, a.s.,
and that $\Xi$ is \emph{recurrent} if $\liminf_{n \to \infty} \xi_n \leq r_0$
for some deterministic $r_0 \in \RP$. For $a \in \RP$ define $\tau_a := \min \{n \in \ZP : \xi_n \leq a\}$.
If $\Xi$ is recurrent then $\Pr_x [ \tau_a < \infty ] = 1$ for any $a > r_0$ and any $x$;
if, moreover, $\Exp_x [\tau_a] < \infty$ for all $a$ sufficiently large and any $x$, 
we say that $\Xi$
is \emph{positive recurrent}, while if for all $a$ sufficiently large
and all $x>a$ we have $\Exp_x [\tau_a] = \infty$, then we say that $\Xi$ is \emph{null recurrent}.
We typically (but not always) assume the following.
\begin{description}
\item\namedlabel{ass:non-confinement-on-RP}{N} We have $\limsup_{n \to \infty} \xi_n = +\infty$, a.s.
\item\namedlabel{ass:tails-on-RP}{$\T$}
There exist constants $\alpha \in (1,2)$, $\beta > \alpha$, $c > 0$, and $x_0 \in \RP$ such that
\[
\lim_{y\to\infty}\sup_{x \geq x_0} \big|y^\alpha \Pr_x[\theta_+ > y] -c \big| = 0, 
 \text{ and } 
 \sup_{x \geq x_0}
  \Exp_x[\theta_-^{\beta}] < \infty. \]
\end{description}
Assumption~\eqref{ass:tails-on-RP} includes a precise version of~\eqref{eq:tail}, and means that  the jumps to the right are
`heavier' than those to the left.
Assumption~\eqref{ass:non-confinement-on-RP}
is \emph{non-confinement} and rules out uninteresting cases;
it is satisfied if, for example (a) $\Xi$ is an irreducible Markov chain on
a locally finite $\X$ (see e.g.~Corollary~2.1.10 in~\cite{mpow}); (b) for some $\eps >0$ the ellipticity condition $\inf_{x \geq 0} \Pr_x [ 
 \theta  \geq \eps ] \geq \eps$ 
holds (Proposition~3.3.4 in~\cite{mpow}); or (c) in~\eqref{ass:tails-on-RP} we have $x_0 =0$ (Proposition~5.2.1 of~\cite{mpow}).

We present a recurrence classification for $\Xi$
which is  an analogue in the heavy-tailed setting
of Lamperti's classification~\cite{lamp1,lamp3}. Recall from~\eqref{eq:mu} that $\mu(x) := \Exp_x [ \theta]$.

\begin{theorem}  \phantomsection
\label{thm:recurrence-on-RP}
\begin{enumerate}[(i)]
\item  Suppose that~\eqref{ass:non-confinement-on-RP} and~\eqref{ass:tails-on-RP} hold. Then $\Xi$ is recurrent if
\begin{equation}
\label{eq:drift-recurrence} \limsup_{x \to \infty} ( x^{\alpha -1} \mu (x) ) < c \pi | \cosec ( \pi \alpha ) | ,\end{equation}
and $\Xi$ is transient if
\begin{equation}
\label{eq:drift-transience} \liminf_{x \to \infty} ( x^{\alpha -1} \mu (x) ) > c \pi | \cosec ( \pi \alpha ) | .\end{equation}
\item The chain $\Xi$ is positive recurrent if there exist $\alpha \in (1,2)$ and $\gamma < \alpha -1$
such that $\limsup_{x \to \infty} \Exp_x [ | \theta |^\alpha ] < \infty$
and $\limsup_{x \to \infty} ( x^\gamma \mu (x) ) < 0$.
\end{enumerate}
\end{theorem}
\begin{remarks}
(i) Note that $\cosec (\pi \alpha ) <0$ for $\alpha \in (1,2)$, so
part~(i) says that the case $\mu (x) \leq 0$ is recurrent, as one would expect,
and shows what magnitude of positive drift must be added to achieve transience.

(ii) Part~(i) says that the critical case where $\mu (x) \sim b x^{1-\alpha}$
for some $b < c \pi | \cosec (\pi \alpha) |$ is recurrent. In fact, this case is null-recurrent by Theorem~\ref{thm:moments-on-RP}(ii)
below.

(iii) For all our results, the Markov property is not essential, and can be dispensed~with
without complicating the proofs, only the notation: cf.~\cite{lamp1} or Chapter~3 of~\cite{mpow}.
\end{remarks}

For the rest of the results that we present,
we will also assume an asymptotic form for the drift.
This is partly for simplicity of statement, but also necessary
to get sharp transitions in our existence-of-moments results.
The assumption is as follows.

\begin{description}
\item\namedlabel{ass:drift}{$\text{D}_{\gamma,b}$}
There exist constants $\gamma \geq 0$ and $b \in \R$ such that
  $\lim_{x \to \infty}( x^\gamma \mu (x) ) = b$.
\end{description}

In the recurrent case, we can precisely quantify the recurrence by showing which moments of passage times
exist.  Here and elsewhere $\Gamma$ denotes the (Euler) gamma function.

\begin{theorem}
\label{thm:moments-on-RP}
Suppose that the Markov chain $\Xi$ on  $\RP$~satisfies~\eqref{ass:tails-on-RP} and~\eqref{ass:drift}.  Let $q >0$.
 Then for all $a$ sufficiently large and all $x>a$, the following hold.
\begin{enumerate}[(i)]
\item If $\gamma > \alpha - 1$ then $\Exp_x[\tau_a^q] < \infty$ for $q<1/\alpha$, and $\Exp_x[\tau_a^q] = \infty$ for $q >1/\alpha$.
\item If $\gamma = \alpha -1$ and $b < c \pi | \cosec(\pi\alpha)|$, then
  $\Exp_x[\tau_a^q] < \infty$ for $q<\nu^\star/\alpha$, and $\Exp[\tau_a^q] = \infty$ for $q>\nu^\star/\alpha$,
  where $\nu^\star := \nu^\star (\alpha,b,c)$ is the unique solution in $(0,\alpha)$
  to
\begin{align}\label{eqn:nu-star}
\frac{b}{c} & = (\nu^\star-1)\frac{\Gamma(1-\alpha)\Gamma(\alpha-\nu^\star)}{\Gamma(2-\nu^\star)}.
\end{align}
\item If $\gamma < \alpha-1$ and $b<0$ then $\Exp_x[\tau_a^q] < \infty$ for $q<
  \frac{\alpha}{\gamma+1}$, and $\Exp_x[\tau_a^q] = \infty$ for $q \geq
    \frac{\alpha}{\gamma+1}$.
\end{enumerate}
\end{theorem}
\begin{remarks}
(i) In parts~(i) and~(ii), we always have null recurrence ($\Exp_x [ \tau_a ] =\infty$),
while in part~(iii) we have positive recurrence ($\Exp_x [ \tau_a ] < \infty$).

(ii) If $\Xi$ is a martingale satisfying assumptions~\eqref{ass:non-confinement-on-RP} and~\eqref{ass:tails-on-RP},
 then~\eqref{ass:drift} is automatically satisfied (for $b = 0$ and any~$\gamma$).
Then $\Xi$ is recurrent
by Theorem~\ref{thm:recurrence-on-RP}(i), and Theorem~\ref{thm:moments-on-RP}(i) shows that
 the critical exponent for $\tau_a$ is $1/\alpha \in (1/2,1)$;
in the case of a spatially homogeneous random walk,
this fact is essentially contained in a result of Doney~\cite{doney}. In this sense, such a martingale with heavy tails away from the origin is \emph{more} recurrent than
simple symmetric random walk on $\RP$, but still null recurrent.
\end{remarks}

\subsection{Walks on the whole line: heavier outward tails}
\label{sec:R-results1}

Now we turn to the case $\X \subseteq \R$.
We first suppose that
for $x>0$, essentially the same conditions~\eqref{ass:tails-on-RP} and~\eqref{ass:drift} as above hold,
while for $x <0$ the symmetric version of those conditions holds, so that 
the outwards increment always has a heavier tail than the
inwards one. This case closely corresponds to the walks on $\RP$ of
the previous section. 
In full, the assumptions are as follows. We extend the definition of $\mu(x) := \Exp_x [ \theta]$ to $x \in \R$.
\begin{description}
\item\namedlabel{ass:non-confinement-on-R}{$\text{N}'$} We have $\limsup_{n \to \infty} |\xi_n| = +\infty$, a.s.
\item\namedlabel{ass:tails-on-R-out}{$\Tout$} 
There exist constants $\alpha \in (1,2)$, $\beta > \alpha$, $c > 0$, and $x_0 \in \RP$ such that
\begin{align*}
& {} \lim_{y \to \infty} \sup_{x \geq x_0} \big| y^\alpha \Pr_x[\theta_+ > y] -c \big|  = \lim_{y \to \infty}
  \sup_{x \leq -x_0} \big| y^\alpha \Pr_x[\theta_- > y] - c \big| = 0, \\
& {}	\sup_{x \geq x_0} \Exp_x[\theta_-^\beta] < \infty, \text{ and } \sup_{x \leq - x_0} \Exp_x[\theta_+^\beta] < \infty . \end{align*}
\item\namedlabel{ass:R-drift}{$\text{D}'_{\gamma,b}$} There exist constants $\gamma \geq 0$ and
  $b \in \R$ such that
\[
\lim_{x \to +\infty}( x^\gamma \mu (x) ) = \lim_{x \to -\infty} ( -|x|^\gamma
\mu (x) ) = b.
\]
\end{description}

For $\Xi$ on $\R$ we now say that $\Xi$ is transient if $\lim_{n \to \infty} | \xi_n | = \infty$, a.s.,
and recurrent if $\liminf_{n \to \infty} | \xi_n | \leq r_0$, a.s.,
for some $r_0 \in \RP$. Setting $\tau_a := \min \{n \in \ZP : |\xi_n| \leq a\}$,
positive and null recurrence are defined analogously to the case of $\RP$.
With transience, however, there are two distinct possibilities.
We say that $\Xi$ is \emph{oscillatory transient} if, a.s.,
\[ \lim_{n \to \infty}  | \xi_n | = \infty,\;\;\text{and}\;\;{-\infty} = \liminf_{n \to \infty} \xi_n < \limsup_{n \to \infty} \xi_n = +\infty .\]
On the other hand, we say that $\Xi$ is \emph{directional transient}
if $\lim_{n \to \infty} \xi_n \in \{ -\infty, +\infty \}$, a.s.
In the following theorem only directional transience appears;
we will see instances of oscillatory transience in the next sections.


\begin{theorem}\label{thm:recurrence-on-R-out}
Suppose that the Markov chain $\Xi$ on $\R$~satisfies~\eqref{ass:non-confinement-on-R},
\eqref{ass:tails-on-R-out}, and~\eqref{ass:R-drift}.  Then the following classification holds.
\begin{enumerate}[(i)]
\item If $\gamma > \alpha-1$ then $\Xi$ is null recurrent.
\item If $\gamma < \alpha-1$ and $b < 0$  then $\Xi$ is positive recurrent.
\item If $\gamma < \alpha-1$ and $b > 0$ then $\Xi$ is directional transient.
\item If $\gamma = \alpha-1$ and $b < c\pi | \cosec(\pi\alpha) |$ then $\Xi$ is null recurrent.
\item If $\gamma = \alpha-1$ and $b > c\pi |\cosec(\pi\alpha) |$ then $\Xi$ is directional transient.
\end{enumerate}
\end{theorem}

The moments result here is essentially the same as that on $\RP$ from Theorem~\ref{thm:moments-on-RP}.

\begin{theorem}
\label{thm:moments-on-R-out}
Suppose that the Markov chain $\Xi$ on  $\R$~satisfies~\eqref{ass:tails-on-R-out} and~\eqref{ass:R-drift}. 
 Then the statements in Theorem~\ref{thm:moments-on-RP} hold
verbatim, given that $\tau_a = \min \{n \in \ZP : |\xi_n| \leq a\}$.
\end{theorem}

\subsection{Walks on the whole line: heavier inward tails}
\label{sec:R-results2}

Again with $\X \subseteq \R$,
we now flip things around  so that the \emph{inwards} increment has the heavier tail.
Here more interesting phenomena occur.
Our assumptions   are as follows.

\begin{description}
\item\namedlabel{ass:tails-on-R-in}{$\Tin$} 
There exist constants $\beta \in (1,2)$, $\alpha > \beta$, $c > 0$, and $x_0\in\RP$ such that
\begin{align*}
& {} \lim_{y \to \infty} \sup_{x \geq x_0} \big| y^\beta \Pr_x[\theta_- > y] - c \big| =
\lim_{y\to \infty} \sup_{x \leq -x_0} \big| y^\beta \Pr_x[\theta_+ > y] - c \big| = 0, \\
& {}	\sup_{x \geq x_0}   \Exp_x[\theta_+^\alpha] < \infty, \text{ and } \sup_{x \leq - x_0} \Exp_x[\theta_-^\alpha] < \infty . \end{align*}
\end{description}
We also assume~\eqref{ass:R-drift}, and,
where necessary, the following additional condition.
\begin{description}
\item\namedlabel{ass:inwards-tail+}{$\text{T}_\delta^+$} 
With $\beta$, $c$, and $x_0$ as in~\eqref{ass:tails-on-R-in}, there exists a constant
  $\delta > 0$ such that
	\[ \sup_{x \geq x_0} \big| y^\beta \Pr_x[\theta_- > y] - c \big| =
  O(y^{-\delta}), \text{ and } \sup_{x \leq -x_0} \big| y^\beta \Pr_x[\theta_+ > y] - c \big| = O(y^{-\delta}). \]
\end{description}

Our recurrence classification in this case is as follows.

\begin{theorem}\label{thm:recurrence-on-R-in} Suppose that the Markov chain $\Xi$ on $\R$~satisfies~\eqref{ass:non-confinement-on-R}, 
\eqref{ass:tails-on-R-in}, and~\eqref{ass:R-drift}.  
Then the following recurrence/transience criteria hold.
\begin{enumerate}[(i)]
\item If $\gamma > \beta-1$ and $\beta > 3/2$ then $\Xi$ is null recurrent.
\item If $\gamma < \beta-1$ and $b < 0$  then $\Xi$ is positive recurrent.
\item If $\gamma < \beta-1$ and $b > 0$ then $\Xi$ is directional transient.
\item If $\gamma = \beta-1$ and $b + c\pi\cot(\pi\beta) < 0$ then $\Xi$ is null recurrent.
\end{enumerate}
If assumption~\eqref{ass:inwards-tail+} is also satisfied, then the following
transience criteria hold.
\begin{enumerate}[(i)]
\setcounter{enumi}{4}
\item If $\gamma > \beta-1$ and $\beta < 3/2$ then $\Xi$ is oscillatory transient.
\item If $\gamma = \beta-1$ and $b + c\pi\cot(\pi\beta) > 0$ then $\Xi$ is oscillatory
 transient. 
\end{enumerate}
\end{theorem}
\begin{remarks}
(i) For zero drift ($b=0$), the phase transition at $\beta = 3/2$
was first observed by Rogozin \& Foss~\cite{rf}, assuming that the law of $\theta$ depends only on the sign of $\xi_0$.

(ii) Note that $\cot(\pi \beta)$ is decreasing on $\beta \in (1,2)$, tends to $+\infty$ as $\beta \downarrow 1$,
to $-\infty$ as $\beta \uparrow 2$, 
and changes sign at $\beta = 3/2$. 
\end{remarks}
 
\begin{theorem}
\label{thm:moments-on-R-in}
Suppose that the Markov chain $\Xi$ on  $\RP$~satisfies~\eqref{ass:tails-on-R-in} and~\eqref{ass:R-drift}.  Let $q >0$.
  Then for all $a$ sufficiently large and all $x>a$, the following hold.
\begin{enumerate}[(i)]
\item If $\gamma > \beta - 1 > 1/2$ then $\Exp_x[\tau_a^q] < \infty$ for $q<\frac{2\beta-3}{\beta}$, and $\Exp_x[\tau_a^q] = \infty$ for $q >\frac{2\beta-3}{\beta}$.
\item If $\gamma = \beta -1$ and $b + c \pi \cot(\pi\beta) < 0$, then
  $\Exp_x[\tau_a^q] < \infty$ for $q< \nu^\star / \beta$, and $\Exp[\tau_a^q] = \infty$ for $q>\nu^\star / \beta$,
  where 
	$\nu^\star := \nu^\star (\beta,b,c)$ is the unique solution in $(0,\beta)$
  to
\begin{align}\label{eqn:nu-star2}
\frac{b}{c} & = \Gamma (\nu^\star ) \left( \frac{(1-\beta+\nu^\star ) \Gamma(1-\beta) }{\Gamma(2-\beta+\nu^\star )} - \frac{\Gamma (\beta-\nu^\star )}{\Gamma(\beta)} \right).
\end{align}
\item If $\gamma < \beta-1$ and $b<0$ then $\Exp_x[\tau_a^q] < \infty$ for $q<
  \frac{\beta}{\gamma+1}$, and $\Exp_x[\tau_a^q] = \infty$ for $q \geq
    \frac{\beta}{\gamma+1}$.
\end{enumerate}
\end{theorem}

\begin{remark}
In the special case $\gamma =0$ and assuming that no jumps away from the origin occur, the conclusion of Theorem~\ref{thm:moments-on-R-in}(iii)
coincides with Theorem~6(i) of~\cite{mpew}.
\end{remark}

\subsection{Walks on the whole line: the balanced case}
\label{sec:R-results3}

For the last of our models on $\R$, the inwards and outwards tail exponents coincide.

\begin{description}
\item\namedlabel{ass:tails-on-R-bal}{$\Tbal$} 
There exist constants $\alpha \in (1,2)$, $c > 0$, and $x_0 \in\RP$ such that
\begin{align*}
\lim_{y \to \infty} \sup_{x \geq x_0} \big| y^\alpha \Pr_x[\theta_+ > y] - c \big| =
\lim_{y \to \infty} \sup_{x \geq x_0} \big| y^\alpha \Pr_x[\theta_- > y] - c \big| & = 0, \text{ and} \\
\lim_{y\to \infty} \sup_{x \leq -x_0} \big| y^\alpha \Pr_x[\theta_+ > y] - c \big| = 
\lim_{y\to \infty} \sup_{x \leq -x_0} \big| y^\alpha \Pr_x[\theta_- > y] - c \big| & =  0 . \end{align*}
\end{description}

We also assume~\eqref{ass:R-drift}, and,
where necessary, an analogue of~\eqref{ass:inwards-tail+}.

\begin{theorem}
\label{thm:recurrence-on-R-bal}
Suppose that the Markov chain $\Xi$ on $\R$~satisfies~\eqref{ass:non-confinement-on-R}, 
\eqref{ass:tails-on-R-bal}, and~\eqref{ass:R-drift}. 
Then the following recurrence/transience criteria hold.
\begin{enumerate}[(i)]
\item If $\gamma > \alpha-1$ then $\Xi$ is null recurrent.
\item If $\gamma < \alpha-1$ and $b < 0$  then $\Xi$ is positive recurrent.
\item If $\gamma < \alpha-1$ and $b > 0$ then $\Xi$ is directional transient.
\item If $\gamma = \alpha-1$ and $b+c\pi\cot(\frac{\pi\alpha}{2})<0$ then $\Xi$ is null recurrent.
\end{enumerate}
If in addition,
the limit assumptions in~\eqref{ass:tails-on-R-bal}
are strengthened to $O(y^{-\delta})$ for some
 $\delta >0$, 
then the following
transience condition also holds.
\begin{enumerate}[(i)]
\setcounter{enumi}{4}
\item  If $\gamma = \alpha-1$ and $b+c\pi\cot(\frac{\pi\alpha}{2})>0$ then $\Xi$ is oscillatory transient.
\end{enumerate}
\end{theorem}

\begin{theorem}
\label{thm:moments-on-R-bal}
Suppose that the Markov chain $\Xi$ on $\R$~satisfies~\eqref{ass:tails-on-R-bal} and~\eqref{ass:R-drift}.  Let $q >0$.
 Then for all $a$ sufficiently large and all $x>a$, the following hold.
\begin{enumerate}[(i)]
\item If $\gamma > \alpha - 1$ then $\Exp_x[\tau_a^q] < \infty$ for $q<1 - \frac{1}{\alpha}$, and $\Exp_x[\tau_a^q] = \infty$ for $q >1 - \frac{1}{\alpha}$.
\item If $\gamma = \alpha -1$ and $b + c \pi \cot(\frac{\pi\alpha}{2}) < 0$, then
$\Exp_x[\tau_a^q] < \infty$
$q< \nu^\star / \alpha$, and $\Exp[\tau_a^q] = \infty$ for $q>\nu^\star / \alpha$,
  where	$\nu^\star := \nu^\star (\alpha,b,c)$ is the unique solution in $(0,\alpha)$
  to
\begin{align*}
\frac{b}{c} & = (\nu^\star-1)\frac{\Gamma(1-\alpha)\Gamma(\alpha-\nu^\star)}{\Gamma(2-\nu^\star)} + \Gamma (\nu^\star) \left( \frac{(1-\alpha+\nu^\star) \Gamma(1-\alpha) }{\Gamma(2-\alpha+\nu^\star)} - \frac{\Gamma (\alpha-\nu^\star)}{\Gamma(\alpha)} \right).
\end{align*}
\item If $\gamma < \alpha-1$ and $b<0$ then $\Exp_x[\tau_a^q] < \infty$ for $q<
  \frac{\alpha}{\gamma+1}$, and $\Exp_x[\tau_a^q] = \infty$ for $q \geq
    \frac{\alpha}{\gamma+1}$.
\end{enumerate}
\end{theorem}

\begin{remarks}
(i) In the special case where the distribution of $\theta$ is symmetric (so the drift is zero), the conclusion of Theorem~\ref{thm:moments-on-R-bal}(i)
coincides with Theorem~4(ii) of~\cite{mpew}.

(ii) 
Sandri\'c \cite{sandricB,sandricT} considers symmetric or near-symmetric increment distributions,
but allows the exponent $\alpha (x)$ to vary with position $x  \in \R$.
In particular, if $\alpha (x) \in (1,2)$ is bounded away from $1$ and $2$, then~\cite{sandricB,sandricT} give
 sufficient conditions for recurrence, transience, and positive recurrence
that generalize, under some additional conditions, the relevant parts of our 
Theorem~\ref{thm:recurrence-on-R-bal} (see the discussion in~\cite[pp.~465--466]{sandricT}).
Similar results for an analogous continuous-time model (Feller process) are obtained in~\cite{sandricS}.
\end{remarks}

\subsection{Walks in higher dimensions: an example}
\label{sec:R2-example}

We present a family of martingales on $\R^2$ with heavy-tailed jumps
that contains both recurrent and transient walks. The example is
deliberately simplistic for expository purposes; the phenomenon could certainly be reproduced for
more naturally motivated random walks.  Similarly, although the walks we describe here are 
2-dimensional, this example can easily be extended to any dimension $d \geq 2$.

Let $\0$ denote the origin of $\R^2$. For $\bx \in \R^2 \setminus \{ \0 \}$
write $\bu_\bx := \bx / \|\bx \|$, and write
$\bv_\bx$ for the unit vector perpendicular to $\bu_\bx$
obtained by rotating $\bu_\bx$ anticlockwise by $\pi/2$.

Our Markov chain $\Xi$ on $\X \subseteq \R^2$
is constructed so that given $\xi_0 = \bx \neq \0$ the jump
$\theta$ is given by
$\theta = \chi \bu_\bx \theta^\cR + (1-\chi) \bv_\bx \theta^\cT$,
where $\chi \in \{0,1\}$,
and $\theta^\cR, \theta^\cT \in \R$ are independent of $\chi$.
If $\chi = 1$ we say that $\theta$ is a \emph{radial} jump, and if $\chi =0$
we say that $\theta$ is a \emph{transverse} jump. 
We also assume that $\Pr_\0 [ \theta = \0 ] <1$.
We suppose that
positive
radial jumps are heavy tailed, that negative radial jumps 
have bounded moments, and that the transverse jumps are symmetric and also have heavy tails.
Specifically, we assume the following.
\begin{enumerate}[(A)]
\item\label{ass:R2-supp} For all $\bx \in \R^2 \setminus \{\0\}$,
 $\Pr_\bx [ \chi = 1 ] = p^\cR \in (0,1)$ and
$\Pr_\bx [ \chi = 0 ] = p^\cT =  1 - p^\cR$.
\item\label{ass:R2-mart}
For all $\bx \in \R^2 \setminus \{ \0\}$, $\Exp_\bx [ \theta^\cR ] = \Exp_\bx [ \theta^\cT ] = 0$.
\item\label{ass:R2-rad-out} There exist constants $\alpha \in (1,2)$ and $c^\cR \in (0,\infty)$ such that
\begin{align*}
\lim_{y \to \infty} \sup_{\bx \in \R^2 \setminus \{ \0\}} \left| y^\alpha \Pr_\bx [ \theta_+^\cR  > y ] - c^\cR \right| = 0 .
\end{align*}
\item\label{ass:R2-rad-in} There exists a constant $\beta > \alpha$ such that
\[
\sup_{\bx \in \R^2 \setminus \{ \0 \}} \Exp_\bx [ (\theta^\cR_- )^\beta ]  < \infty.
\]
\item\label{ass:R2-trans} 
For all $\bx \in \R^2 \setminus \{ \0\}$ and all $y \in \RP$, $\Pr_\bx [ \theta^\cT_- > y] = \Pr_\bx [ \theta^\cT_+ > y]$, and
there exists a constant $c^\cT \in (0,\infty)$ such that
\begin{align*}
\lim_{y \to \infty} \sup_{\bx \in \R^2 \setminus \{ \0\}} \left| y^\alpha \Pr_\bx [ \theta_+^\cT  > y ] - c^\cT \right| = 0 .
\end{align*}
\end{enumerate}

Now  $\Xi$ is transient if $\lim_{n \to \infty} \| \xi_n \| = \infty$, a.s.,
and recurrent if $\liminf_{n \to \infty} \| \xi_n \| \leq r_0$, a.s.,
for some constant $r_0 \in \RP$. Set $\tau_a := \min \{ n \in \ZP  : \| \xi_n \| \leq a \}$.

\begin{theorem}\label{thm:R2-example}
Suppose that the Markov chain $\Xi$ on $\R^2$~satisfies~\eqref{ass:R2-supp}--\eqref{ass:R2-trans}.
 Then the following classification holds.
\begin{enumerate}[(i)]
\item If $p^\cR c^\cR + 2 p^\cT c^\cT \cos(\frac{\pi \alpha}{2}) > 0$, then
  $\Xi$ is null recurrent. 
\item If $p^\cR c^\cR + 2 p^\cT c^\cT \cos(\frac{\pi \alpha}{2}) < 0$, then
  $\Xi$ is transient.
\end{enumerate}
Moreover, in case~(i), for all $a$ sufficiently large we have $\Exp_\bx [ \tau_a^q ] < \infty$ for $q < \nu^\star/\alpha$ and $\Exp_\bx [ \tau_a^q ]
= \infty$ for $q>\nu^\star/\alpha$ and $\|\bx\| > a$, where $\nu^\star$ is the
unique solution in $(0,1)$ to
\begin{equation}\label{eqn:R2-nu-star}
p^\cR c^\cR \frac{\Gamma(\alpha-\nu^\star)\Gamma(1-\alpha)}{\Gamma(1-\nu^\star)} + p^\cT c^\cT
\frac{\Gamma(\frac{\alpha-\nu^\star}{2})\Gamma(1-\frac{\alpha}{2})}{\Gamma(1-\frac{\nu^\star}{2})}
= 0.
\end{equation}
\end{theorem}
\begin{remarks}
(i) The theorem shows how the balance between radial and transverse increments
determines the recurrence behaviour of zero-drift random walks; for walks whose increments have two moments,
the analogous phenomenon is driven by the increment covariance matrix, as described for example in~\cite{gmmw}
or~\cite[\S 4.2]{mpow}.

(ii) Since $\cos(\frac{\pi\alpha}{2}) < 0$ for $\alpha \in (1,2)$, there are walks exhibiting
either behaviour for any $\alpha \in (1,2)$.  Indeed, both the ratio $c^\cR/c^\cT$
and the ratio
$p^\cR/p^\cT$ 
can take
any value in $(0,\infty)$, so for fixed $\alpha \in (1,2)$ the walk is recurrent if either
ratio is taken large enough, and transient if either ratio is taken small enough. 
\end{remarks}

\subsection{Walks in higher dimensions: an open problem}
\label{sec:R2-problem}

We finish this section with an open problem concerning
a partially homogeneous random walk $\Xi$ on $\X = \ZP^2$.
Partition $\ZP^2$ into sets $I := \N^2$,
$A_1 := \N \times \{ 0 \}$,
$A_2 := \{ 0 \} \times \N$, and $O := \{ \0 \}$.
Suppose that $\Pr_\bx [ \theta = \, \cdot \, ]$
is determined only by which of $I, A_1, A_2, O$
contains $\bx$. Write $\theta = (\theta_1, \theta_2)$
in coordinates. Suppose that
$\Pr_\bx [ \theta_j \geq -1 ] =1$ for all $\bx \in I$ and $j \in \{1,2\}$.
Suppose also that $\sup_\bx \Exp_\bx [ \| \theta \| ] < \infty$. 
Denote by
$M_j  = \Exp [ \theta \mid \xi_0 \in A_j ]$
the mean drift vectors on the axes,
and by $M_0 = \Exp [ \theta \mid \xi_0 \in I]$ the drift in the interior.

If $M_0 \neq \0$ the problem  can be essentially classified in terms of  $M_0, M_1$ and $M_2$:
see pp.~39--41 of~\cite{fmm}. The most subtle case has $M_0 = \0$.
If $\sup_\bx \Exp_\bx [ \| \theta \|^p ] < \infty$ for some $p>2$, then 
also important
is the \emph{covariance} of $\theta$ in region $I$: see pp.~56--58 of~\cite{fmm}, or~\cite{aim}.

What about when $M_0 =\0$ but the second moment is infinite? Concretely, 
suppose that for $c_1, c_2 >0$ and $\alpha_1, \alpha_2 \in (1,2)$,
for $j \in \{1,2\}$ and $y \in \N$,
\[ \Pr [ \theta_j = y \mid \xi_0 \in I ] = (c_j + o(1)) y^{-1-\alpha_j} , \text{ as } y \to \infty .\]
Thus the jumps of $\Xi$ are heavy-tailed away from the axes, but bounded towards the axes.

\begin{problem}
\label{problem}
Classify the recurrence and transience of $\Xi$.
\end{problem}

The answer to Problem~\ref{problem}
may well depend on $M_1, M_2$, the $c_j$ and $\alpha_j$, and also on some analogue of
the increment covariance matrix in the heavy-tailed setting.

\section{Semimartingale criteria for real-valued processes}
\label{sec:criteria}

Central in proving our main results 
is the Lyapunov function methodology: we find a function $f$ so that $f(\xi_n)$
satisfies a suitable semimartingale condition, which implies 
the desired properties of $\Xi$. In this section we collect the semimartingale
criteria that we need; mostly this involves presenting known results,
but   a little work is needed to get statements in the form that we want for
establishing directional or oscillatory transience.

All of these results are stated for a real-valued stochastic process $(X_n, n \in \ZP)$
adapted to a filtration $(\cF_n, n \in \ZP)$. This generality, without assuming that $X_n$ is Markov,
is useful so that we can apply the results to e.g.~$X_n = \| \xi_n \|$ for $\xi_n \in \R^2$ the model
in Section~\ref{sec:R2-example}.
First  we present recurrence and transience  criteria
for processes on $\RP$; these results are Theorems~3.5.8 and~3.5.6(ii) in~\cite{mpow}.

\begin{lemma}[Recurrence criterion on $\RP$]
\label{lem:rec-RP} 
Suppose that
  $(X_n)$ is an $(\cF_n)$-adapted process taking values in $\RP$ and
  $\limsup_{n\to\infty} X_n = \infty$, a.s.  Let $f:\RP \to\RP$ be such that
  $f(x) \to \infty$ as $x \to \infty$ and $\Exp f(X_0) < \infty$.  Suppose that there
  exist $x_1 \in \RP$ and $C < \infty$ for which, for all $n \geq 0$,
\begin{align*}
\Exp[ f(X_{n+1})-f(X_n) \mid \cF_n] &\leq 0, \text{ on } \{X_n > x_1 \};\\
\Exp[ f(X_{n+1})-f(X_n) \mid \cF_n] &\leq C, \text{ on } \{X_n \leq x_1 \}.
\end{align*}
Then $\liminf_{n\to\infty} X_n \leq x_1$, a.s.
\end{lemma}

\begin{lemma}[Transience criterion on $\RP$]
\label{lem:trans-RP}
Suppose that
  $(X_n)$ is an $(\cF_n)$-adapted process taking values in $\RP$ and
  $\limsup_{n\to\infty} X_n = \infty$, a.s.  Let $f:\RP \to\RP$ be such that
  $\sup_x f(x) < \infty$, $\lim_{x\to \infty} f(x) = 0$, and $\inf_{y \leq x} f(y) >0$ for
  any $x \in \RP$.  Suppose also that there exists $x_1 \in \RP$ for which, for all
  $n \geq 0$,
\[
\Exp[ f(X_{n+1})-f(X_n) \mid \cF_n] \leq 0, \text{ on } \{X_n > x_1 \}.
\]
Then $\lim_{n\to \infty} X_n = \infty$, a.s.
\end{lemma}

Now we state the following two criteria for processes on $\R$ that apply to
`two-sided' Lyapunov functions. The recurrence criterion
is Lemma~5.3.15 in~\cite{mpow} and the
 transience criterion follows from Lemma~5.3.16 in~\cite{mpow} as in the proof of Theorem~5.3.1 there.

\begin{lemma}[Recurrence criterion on $\R$] 
\label{lem:rec-R}
Suppose that
  $(X_n)$ is an $(\cF_n)$-adapted process taking values in $\R$ and
  $\limsup_{n\to\infty} |X_n| = \infty$, a.s.  Let $f:\R \to\RP$ be such that
  $\lim_{x \to +\infty} f(x) = \lim_{x \to -\infty} f(x) = \infty$ and $\Exp f(X_0) < \infty$.  Suppose that there
  exist $x_1 \in \RP$ and $C < \infty$ for which, for all $n \geq 0$,
\begin{align*}
\Exp[ f(X_{n+1})-f(X_n) \mid \cF_n] &\leq 0, \text{ on } \{|X_n| > x_1 \};\\
\Exp[ f(X_{n+1})-f(X_n) \mid \cF_n] &\leq C, \text{ on } \{|X_n| \leq x_1 \}.
\end{align*}
Then $\liminf_{n\to\infty} |X_n| \leq x_1$, a.s.
\end{lemma}

\begin{lemma}[Transience criterion on $\R$] 
\label{lem:trans-R}
Suppose that
  $(X_n)$ is an $(\cF_n)$-adapted process taking values in $\R$ and
  $\limsup_{n\to\infty} |X_n| = \infty$, a.s.  Let $f:\R \to\RP$ be such that
  $\sup_x f(x) < \infty$, $\lim_{x\to +\infty} f(x) = \lim_{x \to -\infty} f(x) = 0$, and $\inf_{|y| \leq x} f(y) >0$ for
  any $x \in \RP$.  Suppose also that there exists $x_1 \in \RP$ for which, for all
  $n \geq 0$,
\[
\Exp[ f(X_{n+1})-f(X_n) \mid \cF_n] \leq 0, \text{ on } \{|X_n| > x_1 \}.
\]
Then $\lim_{n\to \infty} |X_n| = \infty$, a.s.
\end{lemma}

For processes on $\R$, we also have criteria for directional or oscillatory transience.

\begin{lemma}[Directional transience criterion]
\label{lem:directional-transience}
 Suppose that $(X_n)$ is an $(\cF_n)$-adapted process taking values in $\R$
  and $\limsup_{n\to \infty}|X_n| = \infty$, a.s.  Let $f: \R \to \RP$ be such that
  $\sup_x f(x) < \infty$, $\lim_{x \to +\infty} f(x) = 0$ and $\inf_{y \leq x} f(y) > 0$
  for any $x \in \RP$.  Suppose also that there exists $x_1 \in \RP$ for which for all
  $n\geq 0$,
\begin{align*}
&\Exp[f(X_{n+1}) - f(X_n) \mid \cF_n ] \leq 0, \text{ on } \{ X_n > x_1 \};\\
&\Exp[f(-X_{n+1}) - f(-X_n) \mid \cF_n ] \leq 0, \text{ on } \{ X_n < -x_1 \}.
\end{align*}
Then $\lim_{n\to\infty} X_n \in \{ -\infty, + \infty\}$, a.s.
\end{lemma}

\begin{lemma}[Oscillatory transience criterion]
\label{lem:oscillatory-transience}
Suppose that $(X_n)$ is an $(\cF_n)$-adapted process taking values in $\R$ and $\lim_{n\to
  \infty} |X_n| = \infty$, a.s.  Let $f: \R \to \RP$ be   such that
$\lim_{x \to +\infty} f(x) =  \infty$ and $\Exp f (X_0 ) < \infty$.
Suppose that there exist $x_1 \in \R$ and $C < \infty$ for
  which, for all $n \geq 0$,
\begin{align*}
\Exp[ f(X_{n+1})-f(X_n) \mid \cF_n] &\leq 0, \text{ on } \{X_n > x_1\};\\
\Exp[ f(X_{n+1})-f(X_n) \mid \cF_n] &\leq C, \text{ on } \{ X_n \leq x_1\};\\
\Exp[ f(-X_{n+1})-f(-X_n) \mid \cF_n] &\leq 0, \text{ on } \{X_n < -x_1\};\\
\Exp[ f(-X_{n+1})-f(-X_n) \mid \cF_n] &\leq C, \text{ on } \{ X_n \geq - x_1\}.
\end{align*}
Then $\liminf_{n\to\infty} X_n = -\infty$ and $\limsup_{n\to\infty} X_n =
+\infty$, a.s.
\end{lemma}

We give the proofs of these two results. 
A key step in the proof of Lemma~\ref{lem:directional-transience}
is the following hitting probability estimate.

\begin{lemma}\label{lem:R-hitting-prob}
Suppose that $(X_n)$ is an $\cF_n$-adapted process taking values in $\R$.  Let $f : \R \to
\RP$ be such that $\sup_x f(x) < \infty$ and $\lim_{x\to +\infty} f(x) = 0$.  Suppose that
there exists $x_2 \in \R$ for which $\inf_{y\leq x_2} f(y) > 0$ and, for all $n \geq 0$,
\[
\Exp[f(X_{n+1}) - f(X_n) \mid \cF_n] \leq 0, \text{ on $\{X_n > x_2\}$}.
\]
Then for any $\eps > 0$ there exists $x \in (x_2,\infty)$ for which, for all $n \geq 0$,
\[
\Pr\Big[ \inf_{m\geq n} X_m \geq x_2 \,\Big|\, \cF_n \Big] \geq 1-\eps, \text{ on $\{X_n > x \}$}.
\]
\end{lemma}
\begin{proof}
The idea is standard:~the proof of Lemma~3.5.7 of~\cite{mpow},
although that result is stated for processes on $\RP$,
carries over directly to this case.
\end{proof}
 
\begin{proof}[Proof of Lemma~\ref{lem:directional-transience}]
Under the conditions of lemma, it is clear that the hypotheses of
Lemma~\ref{lem:R-hitting-prob} hold for \emph{any} $x_2$ with $x_2 > x_1 >0$ and for \emph{both}
$(X_n)$ and $(-X_n)$.  So for any $x_2 > x_1$ and $\eps>0$ there exists $x \in
(x_2,\infty)$ for which, for all $n \geq 0$,
\begin{align*}
\Pr \Big[ \inf_{m\geq n} X_m \geq x_2 \Bigmid \cF_n \Big] & \geq 1- \eps, \text{ on $\{X_n > x \}$, and}
\\
\Pr \Big[ \inf_{m\geq n} -X_m \geq x_2 \Bigmid \cF_n \Big] & \geq 1- \eps, \text{ on $\{-X_n > x \}$}.
\end{align*}
In other words, 
\[
\Pr\Big[ \big\{ \inf_{m\geq n} X_m \geq x_2 \big\} \cup \big\{ \sup_{m\geq n} X_m \leq -x_2 \big\}  \Bigmid \cF_n \Big] \geq 1-\eps, \text{ on $\{|X_n| > x \}$}.
\]
Now, let $\sigma_x = \min \{ n \in \ZP : |X_n| > x\}$.  Then, on $\{\sigma_x < \infty \}$,
\[
\Pr\Big[ \big\{ \inf_{m\geq \sigma_x} X_m \geq x_2 \big\} \cup \big\{ \sup_{m\geq \sigma_x} X_m \leq
-x_2 \big\}  \Bigmid \cF_{\sigma_x} \Big] \geq 1-\eps, \as
\]
But 
if $\sigma_x < \infty$ and either $X_m \geq x_2$ for all $m \geq \sigma_x$
or $X_m \leq - x_2$ for all $m \geq \sigma_x$, we have that
 $\{ \liminf_{m \to \infty} X_m \geq x_2
\} \cup \{ \limsup_{m \to \infty} X_m \leq -x_2 \}$ occurs. Thus
\[
\begin{split}
\Pr\Big[ \big\{ \liminf_{m \to \infty} X_m \geq x_2 \big\} &\cup \big\{ \limsup_{m \to \infty} X_m \leq -x_2
\big\}\Big]\\
&\geq \Exp \bigg[ \Pr\Big[ \big\{ \inf_{m\geq \sigma_x} X_m \geq x_2 \big\} \cup \big\{ \sup_{m\geq \sigma_x} X_m \leq
-x_2 \big\}  \Bigmid \cF_{\sigma_x} \Big] \1{\sigma_x < \infty} \bigg]\\
&\geq (1-\eps)\Pr[ \sigma_x < \infty].
\end{split}
\]
Given $\limsup_{n \to \infty} |X_n| = \infty$, a.s., we have $\Pr[\sigma_x <
\infty] = 1$, and since $\eps>0$ was arbitrary,
\[
\Pr\Big[ \big\{ \liminf_{m \to \infty} X_m \geq x_2 \big\} \cup \big\{ \limsup_{m \to \infty} X_m \leq -x_2
\big\}\Big] = 1.
\]
Then, since $x_2 > x_1$ was also arbitrary, with the fact that
$\liminf_{n \to \infty} X_m \geq x$ and
$\limsup_{n \to \infty} X_m \leq - y$ 
are mutually exclusive for $x, y >0$,
the result follows.
\end{proof}

Now we turn to the proof of Lemma~\ref{lem:oscillatory-transience}.
First, we give a variation on Lemma~\ref{lem:rec-RP} for processes
on $\R$; the proof from~\cite[p.~113]{mpow} carries across directly to this setting.

\begin{lemma}
\label{lem:left-recur-R} 
Suppose that $(X_n)$ is an $(\cF_n)$-adapted process taking values in $\R$.
  Let $f:\R \to\RP$ be such that $f(x) \to \infty$ as $x \to +\infty$ and
  $\Exp f(X_0) < \infty$.  Suppose that there exist $x_1 \in \R$ and $C < \infty$ for
  which, for all $n \geq 0$,
\begin{align*}
\Exp[ f(X_{n+1})-f(X_n) \mid \cF_n] &\leq 0, \text{ on $\{X_n > x_1\}$, a.s.};\\
\Exp[ f(X_{n+1})-f(X_n) \mid \cF_n] &\leq C, \text{ on $\{X_n \leq x_1\}$, a.s.}
\end{align*}
Then
\[
\Pr\big[ \{ \limsup_{n \to \infty} X_n < \infty \} \cup \{ \liminf_{n \to \infty} X_n \leq
x_1 \} \big] = 1.
\]
In particular, $\liminf_{n\to\infty} X_n < \infty$, a.s.
\end{lemma}

\begin{proof}[Proof of Lemma~\ref{lem:oscillatory-transience}]
Since $\lim_{n\to\infty} |X_n| = \infty$, a.s., we have that
\[
\Pr\big[ \{ \lim_{n\to \infty} X_n = +\infty \} \cup  \{ \lim_{n\to \infty} X_n = - \infty \} 
\cup \{ \limsup_{n\to\infty} X_n = \limsup_{n \to \infty} (-X_n) = \infty \}\big] = 1.
\]
But Lemma~\ref{lem:left-recur-R} applied to $(X_n)$ implies that
$\liminf_{n\to \infty} X_n < \infty$, a.s., and hence $X_n \not\to +\infty$,
a.s.  Similarly, the lemma applied to $(-X_n)$ implies that
$X_n \not\to -\infty$, a.s., and therefore
$\lim_{n\to \infty} |X_n| = \limsup_{n\to\infty} X_n = \limsup_{n \to \infty}( -X_n) =
\infty$, a.s.
\end{proof}

Finally, we state two results that provide general conditions for the existence and
non-existence of certain moments of passage times. 
These are reformulations of Theorem~1 and Corollary~1 of~\cite{aim}
(see also \cite[\S 6.1]{mpew} or~\cite[\S 2.7]{mpow}).

\begin{lemma}
\label{lem:finite-moments}
Let $Y_n$ be an integrable $\cF_n$-adapted stochastic process, taking values in an
unbounded subset of $\RP$, with $Y_0 = y_0$ fixed.  For $x>0$, let $\sigma_x := \inf\{n\geq
0: Y_n \leq x\}$.  Suppose that there exist $\delta > 0 $, $x>0$ and $\kappa < 1$ such
that for any $n\geq 0$,
\begin{equation}\label{eqn:finite-mom-super}
\Exp[Y_{n+1} - Y_n \mid \cF_n] \leq -\delta Y_n^\kappa, \text{ on $\{n < \sigma_x \}$}.
\end{equation}
Then for any $p \in [0,1/(1-\kappa))$, $\Exp[\sigma_x^p] < \infty$.
\end{lemma}

\begin{lemma}
\label{lem:infinite-moments}
Let $Y_n$ be an integrable $\cF_n$-adapted stochastic process, taking values in an
unbounded subset of $\RP$, with $Y_0 = y_0$ fixed.  For $x>0$, let $\sigma_x := \inf\{n
\geq 0 : Y_n \leq x\}$.  Suppose that there exist $C_1,C_2 > 0$, $x>0$, $p>0$ and $r>1$
such that for any $n\geq0$, on $\{n < \sigma_x \}$ the following hold:
\begin{align}
\Exp[Y_{n+1} - Y_n \mid \cF_n ] &\geq -C_1;\label{eqn:Y-inc-lower-bound}\\
\Exp[Y_{n+1}^r - Y_n^r \mid \cF_n ] &\leq C_2Y_n^{r-1};\label{eqn:Yr-inc-bound}\\
\Exp[Y_{n+1}^p - Y_n^p \mid \cF_n ] &\geq 0.\label{eqn:Yp-submart}
\end{align}
Then for any $q >p$, $\Exp[\sigma_x^q] = \infty$ for $y_0 > x$.
\end{lemma}

\section{Lyapunov function calculations}
\label{sec:lyapunov}

\subsection{Preliminaries}
\label{sec:lyapunov-prelim}

For the case $\X \subseteq \RP$,
we use the Lyapunov function $f_0 : \R_+ \to \R_+$ defined for $\nu \in \R$ by
\[
f_0 (x) := f_0^\nu (x) := \begin{cases}
x^\nu & \text{for $x \geq 1$},\\
1 & \text{for $0 \leq x < 1$}.
\end{cases}
\]
The truncation at $1$ is only necessary for $\nu <0$, but for convenience we define $f_0$ as
above for all $\nu \in \R$. 
For processes on~$\R$ we will use two 
related, but different, extensions of $f_0$ to the whole of $\R$.
These are defined for $\nu \in \R$ as follows.
\[
f_1 (x) := f^\nu_1 (x) := \begin{cases}
{x}^\nu & \text{for $x \geq 1$},\\
1 & \text{for $x < 1$},
\end{cases}
\quad
\text{ and }
\quad
f_2 (x) := f^\nu_2 (x) := \begin{cases}
{|x|}^\nu & \text{for $|x| \geq 1$},\\
1 & \text{for $|x| < 1$}.
\end{cases}
\]
The `two-sided' function $f_2$
will be used to establish recurrence (with $\nu >0$) and transience ($\nu <0$);
the `one-sided' function $f_1$ will be used for 
distinguishing between directional ($\nu < 0$) and oscillatory ($\nu>0$) transience.
Define
\begin{equation}
\label{Di-def}
 D_i (x) :=  D^\nu_i (x) :=\Exp [ f_i (\xi_{n+1} ) - f_i (\xi_n ) \mid \xi_n = x] . \end{equation}
 
Our first estimate for the $D_i$ will be useful when the drift is dominant. In the calculations here and in the rest of the paper,
various constants $C<\infty$ will appear, whose precise value is not important, and may change from line to line.

\begin{lemma}
\label{lem:big-drift}
Suppose that either (i) $\X \subseteq \RP$, or (ii) $\X \subseteq \R$.
In either case, suppose that, for some $\alpha \in (1,2)$,
 $\limsup_{x \to + \infty} \Exp_x [ | \theta |^\alpha ] < \infty$.
Let $\eps >0$.
Then for any $\nu \in (-\eps, \alpha)$, the following asymptotics
hold with $i=0$ in case~(i) and with $i=1$ in case~(ii).
\[ D_i (x) =  \nu x^{\nu-1} \Exp_x[\theta] + O (x^{\nu - \alpha + \eps} ) , \text{ as } x \to +\infty  .\]
\end{lemma}
\begin{proof}
Suppose that either~(i) or (ii) holds, and take $i = 0$ or $i=1$ respectively. Let  $\gamma \in (0,1)$, to be specified later.
By assumption,  $\Exp_x [ | \theta |^\alpha ] \leq C$ for constant $C<\infty$
and all $x$ large enough; suppose that $x \geq 1$ is such an $x$.
For $\nu <0$, $f_i (x) \in [0,1]$ for all $x$ (in $\RP$ or $\R$, as appropriate),
while for $\nu >0$, $f_i(x)$ is non-decreasing.
Moreover, for all $x \geq 1$,
$\theta \geq x^\gamma$ implies that $x + \theta  \leq 2 \theta^{1/\gamma}$. These facts imply the bounds
 \[ \Exp_x [ | f_i (x +\theta) - f_i (x) | \1 { |\theta| \geq x ^\gamma } ]
\leq \begin{cases}
C \Exp_x [  |\theta|^{\nu/\gamma} \1 { |\theta| \geq x ^\gamma } ]
& \text{for } \nu > 0,\\
\Pr_x [ |\theta| \geq x ^\gamma ]& \text{for } \nu < 0.
\end{cases}
\]
Fix $\eps >0$ and $\nu \in (-\eps, \alpha)$.
For $\nu < 0$, Markov's inequality with the moments assumption yields $\Pr_x [| \theta| \geq x ^\gamma ]
= O ( x^{-\alpha \gamma} ) = O (x^{\nu-\alpha+\eps})$,
provided that we take $\gamma > 1 -\frac{\nu+\eps}{\alpha}$,
which we may since $\nu +\eps > 0$. 
If $\nu >0$,
take $\gamma > \frac{\nu}{\alpha}$ so that
\[ 
\Exp_x [  |\theta|^{\nu/\gamma} \1 { |\theta| \geq x ^\gamma } ]
\leq x^{\nu -\alpha\gamma} 
\Exp_x [  |\theta|^{\alpha} ]
= O (x^{\nu-\alpha+\eps}) ,\]
provided that $\gamma > 1 - \frac{\eps}{\alpha}$.
Thus in either case we have
\[ \Exp_x [ | f_i (x +\theta) - f_i (x) | \1 { |\theta| \geq x ^\gamma } ]
= O (x^{\nu-\alpha+\eps}) \]
for a suitable $\gamma \in (0,1)$. On the other hand,
for all $x$ sufficiently large,
\begin{align*}
\Exp_x \left[ ( f_i (x+\theta) - f_i (x) ) \1 { | \theta | < x^\gamma } \right] & = x^\nu \Exp_x \left[ \left( \left(1 + x^{-1} \theta \right)^\nu - 1 \right) \1 { | \theta | < x^\gamma } \right].
\end{align*}
The Taylor expansion
$(1+z)^\nu = 1+\nu z (1 + \phi z)^{\nu -1}$, valid for $z >-1$
and where $\phi = \phi(z) \in [0,1]$,
implies that,
for all $x$ sufficiently large,
\begin{align*}
\left| \Exp_x \left[ ( f_i (x+\theta) - f_i (x) ) \1 { | \theta | < x^\gamma } \right] - \nu x^{\nu - 1} \Exp_x [ \theta  \1 { | \theta | < x^\gamma } ] \right|
\leq C  x^{\nu-2+(2-\alpha)\gamma} \Exp_x [ |\theta|^\alpha ], \end{align*}
which is $o( x^{\nu -\alpha} )$ since $\gamma < 1$.
Noting that $\Exp_x [ | \theta|  \1 { | \theta | \geq x^\gamma } ] \leq x^{\gamma -\alpha\gamma} 
\Exp_x [  |\theta|^{\alpha} ]$ and that
$\gamma -\alpha\gamma < 1 -\alpha +\eps$ 
provided that $\gamma > 1 - \frac{\eps}{\alpha-1}$, the result follows. 
\end{proof}

\subsection{Lyapunov function on the half line}

Lemma~\ref{lem:big-drift} is only useful for large drift;
otherwise, we must use the tail assumptions on the increments
to evaluate more precisely the other contributions to $D_i$.
First we consider $\X \subseteq \RP$; the calculations in this setting
will be a model for the other cases. Set
\begin{equation}
\label{eq:kappa-0}
 \kappa_0 (\nu ) := \kappa_0 (\alpha, \nu) := (1-\nu) \frac{\Gamma (\alpha - \nu) \Gamma (1-\alpha)}{\Gamma (2-\nu) } .\end{equation}
Note that  $\nu \mapsto \kappa_0 (\nu)$ is continuous on $(-\infty, \alpha)$.

\begin{lemma}
\label{lem:lyapunov-f0}
Suppose that the random walk $\Xi$ on $\RP$~satisfies~\eqref{ass:tails-on-RP}.
Then, for any $\nu$ for which $\alpha - \beta< \nu < \alpha$, as $x \to \infty$,
\[
D_0 (x) = \nu x^{\nu-1} \Exp_x[\theta] +
c \nu x^{\nu-\alpha} \kappa_0 (\nu) 
 + o(x^{\nu-\alpha}).
\]
\end{lemma}
\begin{proof}
The case $\nu =0$ is trivially true, since $f_0 (\xi_{n+1})- f_0 (\xi_n)$
is then identically zero.  For $\nu \neq 0$,   the fact that $g(\theta) =
g(\theta_+)+g(-\theta_-)$ for any function with $g(0)=0$ yields
\begin{equation}\label{eqn:posneg}
D_0 (x)
= \Exp_x[f_0 (x+\theta_+) - f_0 (x)] + \Exp_x[f_0 (x-\theta_-) - f_0 (x)].
\end{equation}
In the case $\nu = 1$, for $x \geq 1$ we can write \eqref{eqn:posneg}
as
\begin{align*}
D_0 (x) & = \Exp_x[\theta_+] - \Exp_x[
\theta_-\1{\theta_- \leq x-1}  + (1-x) \1 { \theta_- > x-1 } ]\\
& = \Exp_x[\theta] + \Exp_x[(\theta_--(x-1))\1{\theta_->x-1}],
\end{align*}
and therefore
\[
0 \leq D_0 (x)  - \Exp_x[\theta] \leq \Exp_x[\theta_-\1{\theta_->x-1}].
\]
For $x > x_0$,
from the $\beta$-moments bound in~\eqref{ass:tails-on-RP} and the fact that $\beta > \alpha > 1$, we get
\[\Exp_x[\theta_-\1{\theta_->x-1}] \leq (x-1)^{1-\beta}\Exp[\theta_-^\beta\1{\theta_->x-1}]
\leq C(x-1)^{1-\beta},
\]
and therefore $D_0 (x) =
\Exp_x[\theta] + o(x^{1-\alpha})$, as claimed.
 
Now suppose that $\alpha - \beta < \nu < \alpha$ with $\nu \notin \{ 0,1\}$. 
For any $x \geq 1$,
we can write $\Exp_x[ f_0 (x+\theta_+) - f_0 (x)] = x^\nu\Exp_x[A_1+A_2+A_3]$, where we define the random variables
\begin{align*}
A_1 &= ((1+\theta_+/x)^\nu - 1-\nu\theta_+/x)\1{\theta_+ \leq x^\eps},\\
A_2 &= ((1+\theta_+/x)^\nu - 1 - \nu\theta_+/x)\1{\theta_+ > x^\eps},\\
A_3 &= \nu\theta_+/x.
\end{align*}
where $\eps \in (0, \frac{2-\alpha}{2})$ is fixed.

For $\Exp_x [ A_1]$, we use the Taylor expansion
$(1+z)^\nu = 1+\nu z +\frac12 \nu(\nu-1) z^2 (1+\phi z)^{\nu-2}$ for some $\phi= \phi(z) \in [0,1]$, valid
for $z>-1$, and the fact that $\nu<\alpha<2$ to write
\begin{equation}\label{eqn:Taylor-nu}
| (1+z)^\nu - 1- \nu z | \leq 
\begin{cases}
\frac12|\nu(\nu-1)| z^2 & \text{for $z\geq 0$},\\
\frac12|\nu(\nu-1)| z^2(1+z)^{\nu-2} &\text{for $-1<z<0$}.
\end{cases}
\end{equation}
Hence
\[
|A_1| \leq \frac{|\nu(\nu-1)|\theta_+^2}{2x^2}\1{\theta_+ \leq x^\eps}
\]
and therefore $\Exp_x[A_1] = O(x^{2\eps -2}) = o(x^{-\alpha})$, since $\eps < \frac{2-\alpha}{2}$.

We now show that 
\begin{equation}
\label{eq:A2}
\Exp_x[A_2] = c\nu x^{-\alpha} \int_0^\infty ((1+u)^{\nu-1}-1)u^{-\alpha}\ud u + o(x^{-\alpha});
\end{equation}
the integral here being finite by Lemma~\ref{lem:int-positive-part} (with $p=1-\alpha, q=\nu$).
To get~\eqref{eq:A2}, define $g_x : \RP \to \R$ by $g_x(y) = (1+y/x)^\nu-1-\nu y/x$, which is differentiable
with derivative $g'_x(y) = (\nu/x)((1+y/x)^{\nu-1}-1)$.  Since $g'(y) > 0$ for all $y >0$
if $\nu < 0$ or $\nu >1$,
 and $g'(y) < 0$ for all $y > 0$ if $\nu \in (0,1)$,  
$g_x$ is monotonic, and applying Lemma~\ref{lem:exp-tail} we obtain
\begin{equation}
\label{eq:A2a}
\Exp_x[A_2] = g_x(x^\eps)\Pr_x[\theta_+ > x^\eps] + \int_{x^\eps}^\infty g'_x(y)
\Pr_x[\theta_+ > y]\ud y,
\end{equation}
the integral  being finite for any $x > x_0$, since 
$g'_x$ is continuous and finite over $[x^\eps,\infty)$,
and $g'_x (y) = O (y^{\nu -1 })$ as $y \to \infty$,
so, by the $\alpha$-tail assumption in~\eqref{ass:tails-on-RP}, the integrand decays like
$y^{\nu-\alpha-1}$.
Another Taylor's theorem calculation  
shows   $g_x(x^\eps) = O (x^{2\eps -2} ) = o(x^{-\alpha})$.

We now consider the $x \to \infty$~asymptotics of the integral in~\eqref{eq:A2a}.
First note that, 
\begin{align*}
\left| \int_{x^\eps}^\infty g'_x(y) \Pr_x[\theta_+ > y]\ud y - \int_{x^\eps}^\infty g'_x(y)
  cy^{-\alpha}\ud y \right|
	& \leq \int_{x^\eps}^\infty | g_x' (y) | \left| y^\alpha \Pr_x[\theta_+ > y]\ud y - c \right| y^{-\alpha} \ud y .
	\end{align*}
Then by assumption~\eqref{ass:tails-on-RP}, for any $\delta >0$ we can find $x_1 \in \RP$ such that, for all $x \geq x_1$,
\[ \int_{x^\eps}^\infty | g_x' (y) | \left| y^\alpha \Pr_x[\theta_+ > y]\ud y - c \right| y^{-\alpha} \ud y 
\leq \delta \int_{x^\eps}^\infty | g_x' (y) |  y^{-\alpha} \ud y
\leq \delta \left| \int_{0}^\infty  g_x' (y)   y^{-\alpha} \ud y \right|,
 \]
because $g_x'(y)$ never changes sign.
Also, by Taylor's theorem again, 
\[
\biggl|\int_0^{x^\eps} g'_x(y) y^{-\alpha} \ud y \biggr| = \int_0^{x^\eps} |g'_x(y)| y^{-\alpha} \ud y \leq 
C x^{-2} \int_0^{x^\eps} |\nu(\nu-1)|y^{1-\alpha}\ud y = o(x^{-\alpha})
\]
since  $\alpha < 2$ and $\eps < 1$.
Therefore, 
\[
\left|\int_{x^\eps}^\infty g'_x(y) \Pr[\theta_+ > y] \ud y - \int_0^\infty g'_x(y) c y^{-\alpha}
  \ud y \right| \leq \delta \left| \int_0^\infty g'_x(y)y^{-\alpha}\ud y \right| +  o(x^{-\alpha}) ,
\]
and since $\delta > 0$ was arbitrary, we have
\[ 
\int_{x^\eps}^\infty g'_x(y) \Pr[\theta_+ > y] \ud y  = (c+o(1))\int_0^\infty g'_x(y) y^{-\alpha}
  \ud y + o(x^{-\alpha}) ,
\]
which, after the change of variable $y=ux$, and using~\eqref{eq:A2a},
 yields~\eqref{eq:A2} .

Thus from the fact that $\Exp_x[ f_0 (x+\theta_+) - f_0 (x)] = x^\nu\Exp_x[A_1+A_2+A_3]$, we obtain
\begin{equation}\label{eqn:thetapos}
\Exp_x[f_0 (x+\theta_+)-f_0 (x)] = c \nu x^{\nu-\alpha} \int_0^\infty
((1+u)^{\nu-1} - 1)u^{-\alpha} \ud u + \nu x^{\nu-1} \Exp_x[\theta_+] + o(x^{\nu-\alpha}).
\end{equation}

We now consider $\Exp_x[f_0(x-\theta_-)-f_0(x)]$ in~\eqref{eqn:posneg}. Fix (another) $\eps \in (0,1)$
with $\eps < \frac{\nu+\beta-\alpha}{\beta}$ and $\eps < \frac{\beta-\alpha}{\beta}$;
this choice is possible since $\alpha < \beta$ and $\nu > \alpha - \beta$. Note that, for $\nu <
0$, the Lyapunov function satisfies $f_0(x) \in [0,1]$ for all $x\in \RP$, so that, for $x > x_0$,
\begin{equation}
\label{eq:f0-big-neg}
\left| \Exp_x[(f_0(x-\theta_-)-f_0(x))\1{\theta_- > x^{1-\eps}}] \right| \leq \Pr_x[\theta_- > x^{1-\eps}]
\leq Cx^{-\beta(1-\eps)} = o(x^{\nu-\alpha}),
\end{equation}
by Markov's inequality, the $\beta$-moments assumption in~\eqref{ass:tails-on-RP}, and the choice of $\eps$. On the other hand, if
$\nu >0$, then $f_0$ is non-decreasing and, for $x > x_0$,
\[
\left|\Exp_x[(f_0(x-\theta_-)-f_0(x))\1{\theta_- > x^{1-\eps}}] \right| \leq x^\nu \Pr_x[\theta_- > x^{1-\eps}]
\leq Cx^{\nu-\beta(1-\eps)} = o(x^{\nu-\alpha}),
\]
again, by choice of $\eps$.
It remains to consider the random variable
\[
(f_0(x-\theta_-)-f_0(x))\1{\theta_- \leq x^{1-\eps}} = x^\nu((1-\theta_-/x)^\nu - 1 )\1{\theta_- \leq x^{1-\eps}},
\]
for all $x$ sufficiently large.
Using~\eqref{eqn:Taylor-nu} for $z = -\theta_-/x$, we obtain
\[
\big|(1-\theta_-/x)^\nu - 1 + \nu\theta_-/x \big|\1{\theta_- \leq x^{1-\eps}} \leq \frac{|\nu(\nu-1)|\theta_-^2}{2x^2}(1 + O(x^{-\eps}))\1{\theta_- \leq x^{1-\eps}}.
\]
Assumption~\eqref{ass:tails-on-RP} means that there exists $\beta' \in
(\alpha,2)$ with $\limsup_{x \to \infty} \Exp_x[\theta_-^{\beta'}] < \infty$, so that $\Exp_x[\theta_-^2\1{\theta_- \leq x^{1-\eps}}] \leq
x^{(2-\beta')(1-\eps)}\Exp_x[\theta_-^{\beta'}] = o(x^{2-\alpha})$, by choice of $\beta'$.
 Therefore
\begin{equation}\label{eqn:thetaneg}
\Exp_x[f_0(x-\theta_-)-f_0(x)] = - \nu x^{\nu-1}
\Exp_x[\theta_-\1{\theta_-\leq x^{1-\eps}}] + o(x^{\nu-\alpha}).
\end{equation}
Also, assumption~\eqref{ass:tails-on-RP} yields $\Exp_x[\theta_-
\1{\theta_->x^{1-\eps}}]  \leq x^{(1-\beta)(1-\eps)} \Exp_x [  \theta_-^\beta ]  = o(x^{1-\alpha})$,
since $\eps < \frac{\beta-\alpha}{\beta} < \frac{\beta-\alpha}{\beta-1}$.
Thus combining~\eqref{eqn:thetapos} and~\eqref{eqn:thetaneg},  we get
\[
D_0 (x) = \nu x^{\nu-1} \Exp_x[\theta] +  c \nu x^{\nu-\alpha} \int_0^\infty
((1+u)^{\nu-1} - 1)u^{-\alpha} \ud u  + o(x^{\nu-\alpha}).
\]
Finally, using Lemma~\ref{lem:int-positive-part} with $p=1-\alpha$ and $q=\nu$, the integral $\int_0^\infty
((1+u)^{\nu-1} - 1)u^{-\alpha} \ud u$ is equal to $(1-\nu)\Gamma(\alpha - \nu)\Gamma(1-\alpha)/\Gamma(2-\nu)$
and the result follows.
\end{proof}

\subsection{Lyapunov functions on the whole line}
 
In the case of heavier outwards tails,
the computations for $f_1$ and $f_2$ are naturally related to
those that we did for $f_0$ in the last section,
so we can reuse many calculations here. 

\begin{lemma}
\label{lem:lyapunov-R-out}
Suppose that the random walk $\Xi$ on $\R$~satisfies~\eqref{ass:tails-on-R-out}.
Then, for any $\nu$ for which $\alpha - \beta< \nu < \alpha$, the following hold.
First, as $x \to +\infty$,
\[
D_1 (x)
= \nu x^{\nu-1} \Exp_x[\theta] +
c \nu x^{\nu-\alpha} \kappa_0 (\nu)
+ o(x^{\nu-\alpha}),
\]
where $\kappa_0$ is defined at~\eqref{eq:kappa-0}.
Second,  as $x \to \pm \infty$,
\begin{equation}
\label{eq:f2-increment}
D_2 (x) =
 \nu \sign(x) |x|^{\nu-1} \Exp_x[\theta] +
c \nu |x|^{\nu-\alpha} \kappa_0 (\nu)
+ o(|x|^{\nu-\alpha}).
\end{equation}
\end{lemma}
\begin{proof}
Suppose that $x \geq 1$. Here the relevant part of assumption~\eqref{ass:tails-on-R-out}
coincides with the assumption~\eqref{ass:tails-on-RP} used in the proof of Lemma~\ref{lem:lyapunov-f0}.
Moreover, since $f_0(x) = f_1(x) = f_2(x)$ for all $x \geq 0$ (and a fixed $\nu$),
conditional on $\xi_n = x$ it is clearly the case that
 $f_i ( x + \theta_+ )$ and $f_i (x-\theta_-) \1 { \theta_- \leq x }$
do not depend on which $i \in \{0,1,2\}$ we are using. Thus the only
difference from our computation in Lemma~\ref{lem:lyapunov-f0}
arises from the possibility now that $\theta_- > x$
(which was previously precluded).

In the places in the proof of Lemma~\ref{lem:lyapunov-f0}
where $\theta_-$ is allowed to be big,
we used only that (i) $f_0 (x) = 1$ for $x \leq 1$, (ii) $f_0 (x) \in [0,1]$ for all $x \in \RP$ if $\nu < 0$,
and (iii) $f_0(x)$ is non-decreasing for $x \in \RP$ if $\nu > 0$. All of (i)--(iii) extend to $x \in \R$ with $f_1$ in place of $f_0$. Thus the proof of the result for $f_1$
follows verbatim that of Lemma~\ref{lem:lyapunov-f0},   replacing $f_0$ by $f_1$,
and  noting that some statements should be extended from $\RP$ to $\R$.

Consider $f_2$. 
Suppose that we can show that~\eqref{eq:f2-increment}
holds for $x \to +\infty$.
If assumption~\eqref{ass:tails-on-R-out}
holds for $\xi_n$, it also holds for $-\xi_n$.
 Then by the symmetry $f_2 (-x ) \equiv f_2 (x)$, we may apply~\eqref{eq:f2-increment} to the process $-\xi_n$,
to get that
$D_2 (-x) = 
\Exp [ f_2 ( - \xi_{n+1} ) - f_2 (- \xi_n ) \mid -\xi_n = x]$
is, as $x \to +\infty$, equal to the right-hand side of~\eqref{eq:f2-increment}
but with $\Exp_x [ \theta ]$ replaced by
$-\Exp_{-x} [ \theta]$. This shows that~\eqref{eq:f2-increment} also holds for
$x \to -\infty$. Thus it suffices to prove~\eqref{eq:f2-increment} for $x \to +\infty$; so
 we take $x \geq 1$, as in the first paragraph of this proof.

	We describe how to modify the proof of Lemma~\ref{lem:lyapunov-f0} to obtain~\eqref{eq:f2-increment}.
 For $\nu=1$, the analogue of~\eqref{eqn:posneg} is
\begin{align*}
D_2 (x)  
& = \Exp_x [ \theta_+ ] - \Exp_x [ \theta_-\1{\theta_-\leq x-1} ]
+ \Exp_x [ (1-x) \1 {  x-1<\theta_-\leq x+1 } ] \\
& {} \quad 
+ \Exp_x [ (\theta_- - 2x) \1 { \theta_- > x +1 } ] \\
& = \Exp_x[\theta] + \Exp_x[(\theta_--(x-1))\1{\theta_->x-1}] + \Exp_x[(\theta_--(x+1))\1{\theta_->x+1}]  .
\end{align*}
It follows that $0 \leq D_2 (x) - \Exp_x[\theta] \leq 2\Exp_x[\theta_-\1{\theta_->x-1}]$,
and so $D_2(x) = \Exp_x[\theta] +o(x^{1-\alpha})$ as before.

Now suppose that $\alpha-\beta< \nu< \alpha$ with $\nu \notin \{ 0,1 \}$.  Following the proof of
Lemma~\ref{lem:lyapunov-f0} exactly, we can show that equation~\eqref{eqn:thetapos} holds for $f_2$
in place of $f_0$.  Now
consider $\Exp_x[f_2(x-\theta_-)-f_2(x)]$.  The random variable
$(f_2(x-\theta_-)-f_2(x))\1{\theta_-\leq x^{1-\eps}}$ can be dealt with in exactly the same
way as in the proof of
Lemma~\ref{lem:lyapunov-f0}, so to show that the analogue of equation~\eqref{eqn:thetaneg} holds for $f_2$, it is enough to
prove that, as $x \to +\infty$,
\begin{equation}
\label{eqn:thetaneg-large}
\Exp_x[(f_2(x-\theta_-)-f_2(x))\1{\theta_- > x^{1-\eps}}] = o(x^{\nu-\alpha}).
\end{equation}
If $\nu < 0$ this follows in the same way as~\eqref{eq:f0-big-neg}, since $f_2(x) \in [0,1]$ for all $x\in \R$.
If $\nu >0$, then $f_2(x-\theta_-) \leq f_2(x)$ whenever $\theta_- \leq 2x$, so we have
\[
\left| \Exp_x[(f_2(x-\theta_-)-f_2(x))\1{x^{1-\eps} < \theta_- \leq 2x}] \right| \leq x^\nu \Pr_x[\theta_- > x^{1-\eps}]
\leq Cx^{\nu-\beta(1-\eps)} = o(x^{\nu-\alpha}),
\]
for small enough $\eps>0$ because $\beta > \alpha$. Since $f_2(x-\theta_-) - f_2(x) \leq
\theta_-^\nu$ if $\theta_- > 2x$, we have
\[
\left|\Exp_x[(f_2(x-\theta_-)-f_2(x))\1{ \theta_- > 2x}] \right| \leq \Exp_x[\theta_-^\nu\1{\theta_->2x}]
\leq Cx^{\nu-\beta} = o(x^{\nu-\alpha})
\]
because $\beta > \alpha$. This proves~\eqref{eqn:thetaneg-large}, and hence also the analogue of~\eqref{eqn:thetaneg} for $f_2$.
  The remainder of the proof exactly follows the proof of
Lemma~\ref{lem:lyapunov-f0}.
\end{proof}

Now we turn to the case of heavier inwards tails. Define
\begin{align}
\label{eq:kappa-1}
\kappa_1 (\nu ) & := \kappa_1 (\beta, \nu) := - \frac{(1-\beta+\nu) \Gamma (\nu+1) \Gamma ( 1-\beta)}{\Gamma (2-\beta + \nu)} ,
\text{ and}\\
\label{eq:kappa-2}
\kappa_2 (\nu) & := \kappa_2 (\beta, \nu) := \Gamma (\nu) \left( \frac{ \Gamma (\beta-\nu)}{\Gamma (\beta)}  - \frac{(1-\beta+\nu) \Gamma ( 1-\beta)}{\Gamma (2-\beta + \nu)}
\right). \end{align}

\begin{lemma}
\label{lem:lyapunov-R-in}
Suppose that the random walk $\Xi$ on $\R$~satisfies~\eqref{ass:tails-on-R-in}.
Then, for any $\nu$ for which $0 < \nu < \beta$, the following hold.
(i) As $x \to +\infty$,
\[
D_1 (x)
= \nu x^{\nu-1} \Exp_x[\theta] +
c x^{\nu -\beta} \kappa_1 (\nu)
+
o(x^{\nu-\beta}).
\]
(ii) As $x \to \pm\infty$,
\begin{align*}
D_2 (x) &= \nu \sign(x) |x|^{\nu-1} \Exp_x[\theta] +
c \nu |x|^{\nu-\beta} \kappa_2 (\nu) +
o(|x|^{\nu-\beta}).
\end{align*}
If in addition, assumption~\eqref{ass:inwards-tail+} is satisfied for some $\delta >0$, 
then the asymptotic expression in~(ii) also holds for $-\delta_\star < \nu < 0$, where $\delta_\star = \min\{\delta, 1\}$.
\end{lemma}
\begin{proof}
As in the proof of Lemma~\ref{lem:lyapunov-R-out}, for both results it suffices to suppose that $x \geq 1$
(and later to take $x\to+\infty$). Also, without loss of generality, we may suppose that 
$1 < \beta < \alpha < 2$. We can treat $f_1$ and $f_2$ together for most of the computations.
Indeed, let $i \in \{1,2\}$. Then, by the monotonicity of $x \mapsto f_i(x)$ for $x \geq 1$,
\[
\Exp_x \left[\left| f_i (x+\theta_+)-f_i (x) \right| \1{\theta_+ \geq x} \right]  \leq \begin{cases}
\Exp_x[ (2\theta_+)^\nu \1{\theta_+ \geq x}] & \text{for $\nu > 0$}, \\
x^\nu \Pr_x[\theta_+ \geq x] & \text{for $\nu < 0$}.
\end{cases}
\]
Suppose that $\nu < \beta$. Markov's inequality with assumption~\eqref{ass:tails-on-R-in}
implies that, for $x >x_0$, $\Pr_x[\theta_+ \geq x] \leq C
x^{-\alpha} = o(x^{-\beta})$, and since $\nu < \beta<\alpha$ we also have
\[
\Exp_x[ \theta_+^\nu \1{\theta_+ \geq x}] = x^{\nu-\alpha}\Exp_x[\theta_+^\alpha \1{\theta_+
  \geq x}] \leq C x^{\nu-\alpha} = o(x^{\nu-\beta}).
\]
Now consider $\Exp_x[ | f_i(x+\theta_+)-f_i(x)- \nu x^{\nu-1} \theta_+| \1{\theta_+ < x}]$.
Using~\eqref{eqn:Taylor-nu}, and the fact that~\eqref{ass:tails-on-R-in} holds for some $\alpha \in (1,2)$, 
we can bound this expression by
\[
x^\nu \Exp_x\left[\frac{|\nu(\nu-1)|\theta_+^2}{2x^2} \1{\theta_+ < x}\right] \leq C
x^{\nu-\alpha} \Exp_x[\theta_+^\alpha \1{\theta_+ < x}] = o(x^{\nu-\beta}).
\]
Since $x^{\nu-1}\Exp_x[\theta_+ \1{\theta_+ \geq x}] \leq x^{\nu-1} \cdot C x^{1-\alpha} =
  o(x^{\nu-\beta})$, these results combine to give
\begin{equation}\label{eqn:heavy-in-thetapos}
\Exp_x[ f_i (x+\theta_+)-f_i (x) ] = \nu x^{\nu-1} \Exp_x[\theta_+] + o(x^{\nu-\beta}).
\end{equation}

For the expectation $\Exp_x[ f_i(x-\theta_-) - f_i(x) ]$, we consider the random variables
\begin{align*}
B_{i,1} &= ( f_i(x-\theta_-) - f_i(x) + \nu x^{\nu-1} \theta_-)\1{\theta_- \leq x^\eps},\\
B_{i,2} &= ( f_i(x-\theta_-) - f_i(x) + \nu x^{\nu-1} \theta_-)\1{\theta_- > x^\eps},\\
B_3 &= - \nu x^{\nu-1} \theta_-,
\end{align*}
where $\eps \in (0, \frac{2-\beta}{2} )$ is fixed.
For $x \geq 1$ we can write 
\[ |B_{i,1}| = x^\nu \big| (1-\theta_-/x)^\nu - 1 + \nu \theta_-/x \big| \1{\theta_-
  \leq x^\eps}, \]
	and using~\eqref{eqn:Taylor-nu} with $z=-\theta_-/x$ (valid since $\nu <
\beta < 2$ and $-\theta_-/x \geq -x^{\eps-1} > -1$ for large enough $x$),
 we find that for $i \in \{1,2\}$, for all large enough $x$,
\[
|B_{i,1}| \leq \frac{1}{2}x^{\nu-2}|\nu(\nu-1)|\,\theta_-^2(1-\theta_-/x)^{\nu-2}\1{\theta_-
  \leq x^\eps} \leq | \nu (\nu-1) | x^{\nu+2\eps-2},
\]
and therefore $\Exp_x[B_{i,1}] = o(x^{\nu-\beta})$ by choice of $\eps$.

It is in the calculation of $\Exp_x[B_{i,2}]$ that we see the difference between the $f_i$.
Define the function $h_{i,x} : \RP \to \R$ by $h_{i,x}(y) =
f_i(x-y)-f_i(x) + \nu x^{\nu-1} y$, which, as a function of $y$, is continuous on $\RP$ and 
only fails to be differentiable at $y=x- 1$ and, in the case of $i=2$, also at $y = x+1$.
Away from these two points, the derivative is
\begin{equation}
\label{eq:h-prime}
h'_{i,x}(y) = \begin{cases}
-\nu(x-y)^{\nu-1} + \nu x^{\nu-1} &\text{for $0 < y < x-1$},\\
 \nu x^{\nu-1} &\text{for $x-1 < y< x+1$},\\
(i-1) \nu(y-x)^{\nu-1} + \nu x^{\nu-1} &\text{for $y > x+1$}.
\end{cases}
\end{equation}
Hence $h_{i,x}$ is piecewise monotonic, and we can apply Lemma~\ref{lem:exp-tail} to get
\begin{equation}
\label{eq:B2}
\Exp_x[B_{i,2}] = h_{i,x} (x^\eps)\Pr_x[\theta_- > x^\eps] + \int_{x^\eps}^\infty h'_{i,x}(y)
\Pr_x[\theta_- > y] \ud y,
\end{equation}
and the integral is finite for fixed $x > x_0$, since as $y\to \infty$ the integrand decays like $y^{-\beta}$ 
(if $i=1$, or if $i=2$ and $\nu \leq 1$)
or $y^{\nu-\beta-1}$ (if $i=2$ and $\nu > 1$). 
Using~\eqref{eqn:Taylor-nu} we find that $h_{i,x}(x^\eps) = O ( x^{\nu+2\eps-2} ) = o(x^{\nu-\beta})$.

Consider the integral in~\eqref{eq:B2}. Let $\delta (y) = \sup_{x \geq x_0} \sup_{z \geq y} | z^\beta \Pr_x [ \theta_- > z ] - c |$;
then $\delta (y)$ is non-increasing and, by assumption~\eqref{ass:tails-on-R-in}, $\delta(y) \to 0$ as $y \to \infty$.
From~\eqref{eq:h-prime}, using Taylor's theorem and the fact that  $\beta <2$, for any $\gamma \in (0,1)$ we obtain
\[
\left|\int_0^{x^\gamma} h'_{i,x}(y) y^{-\beta} \ud y \right| = \int_0^{x^\gamma} |h'_{i,x}(y)| y^{-\beta} \ud y
\leq 2x^{\nu-2}|\nu(\nu-1)|\int_0^{x^\gamma} y^{1-\beta} \ud y = o(x^{\nu-\beta}),
\]
and the same bound holds with $\Pr_x[\theta_->y]$ in place of $y^{-\beta}$, provided $x > x_0$.
Hence, for any $\kappa \in (\eps,1)$,
\begin{align*}
& {} \left| \int_{x^\eps}^\infty h'_{i,x}(y) \Pr_x[\theta_->y] \ud y   -  \int_{0}^\infty
  h'_{i,x}(y)cy^{-\beta}\ud y \right|\\
& {} \qquad\qquad {} \leq \left| \int_{x^\kappa}^\infty h'_{i,x}(y) \Pr_x[\theta_->y] \ud y   -  \int_{x^\kappa}^\infty
  h'_{i,x}(y)cy^{-\beta}\ud y \right| + o (x^{\nu -\beta}) .
\end{align*}
By~\eqref{eq:h-prime}, the sign of $h'_{i,x}(y)$
is  constant on $(0,x-1)$ and on $(x-1,\infty)\setminus \{x+1\}$; thus
\begin{align*}
& {} \left| \int_{x^\kappa}^\infty h'_{i,x}(y) \Pr_x[\theta_->y] \ud y   -  \int_{x^\kappa}^\infty
  h'_{i,x}(y)cy^{-\beta}\ud y \right|\\
& {} \qquad\qquad {} \leq \delta (x^\kappa ) \left( \left|\int_{x^\kappa}^{x-1} h'_{i,x}(y) y^{-\beta} \ud y \right| +
\left|\int_{x-1}^{\infty} h'_{i,x}(y) y^{-\beta} \ud y \right| \right).
\end{align*}
Therefore, for $\kappa \in (\eps,1)$, as $x \to +\infty$,
\begin{align}
\label{eq:in-jump-bound} 
& {} \left| \int_{x^\eps}^\infty h'_{i,x}(y) \Pr_x[\theta_->y] \ud y   -  \int_{0}^\infty
  h'_{i,x}(y)cy^{-\beta}\ud y \right| \nonumber\\
& {} \qquad\qquad {} \leq \delta (x^\kappa ) \left( \left|\int_{0}^{x-1} h'_{i,x}(y) y^{-\beta} \ud y \right| +
\left|\int_{x-1}^{\infty} h'_{i,x}(y) y^{-\beta} \ud y \right| \right) + o(x^{\nu-\beta}).
\end{align}
For both $i \in \{1,2\}$, we get from~\eqref{eq:h-prime} that
for $\nu > -1$, $\nu \neq 0$,
\begin{align*}
\left|\int_{0}^{x-1} h'_{i,x}(y) y^{-\beta} \ud y \right|
  = \left| \nu x^{\nu - \beta} \int_0^{1-x^{-1}} \left( (1-u)^{\nu - 1} - 1 \right) u^{-\beta} \ud u \right|  
  \leq C x^{\nu-\beta} \left( 1 + x^{-\nu} \right) ,
\end{align*}
by an application of Corollary~\ref{cor:int-negative-part-to-x-1} with $p=1-\beta$ and $q=\nu$.
For $i =2$, 
 Lemma~\ref{lem:int-negative-part-from-x+1} (with $p=1-\beta$, $q=\nu$) shows that for $-1 < \nu < \beta$,
$\nu \neq 0$, 
\[  
\left|\int_{x+1}^{\infty} h'_{2,x}(y) y^{-\beta} \ud y \right| 
 = \left| \nu  x^{\nu - \beta} \int_{1+x^{-1}}^\infty \left( (u-1)^{\nu - 1} + 1 \right) u^{-\beta} \ud u \right|
\leq C x^{\nu-\beta} \left( 1 + x^{-\nu} \right) . \]
More straightforward is the case $i=1$, where we have
\[ \left|\int_{x+1}^{\infty} h'_{1,x}(y) y^{-\beta} \ud y \right| 
=
\left|\nu x^{\nu - 1} \int_{x+1}^{\infty}  y^{-\beta} \ud y \right| \leq C x^{\nu -\beta} .\]
Thus from~\eqref{eq:in-jump-bound} we obtain, for $i \in \{1,2\}$,
\[ \left| \int_{x^\eps}^\infty h'_{i,x}(y) \Pr_x[\theta_->y] \ud y   -  \int_0^\infty
  h'_{i,x}(y)cy^{-\beta}\ud y \right| \leq C (1 + x^{-\nu}) \delta (x^\kappa) x^{\nu-\beta} + o ( x^{\nu -\beta}).
 \]
This last bound is $o ( x^{\nu -\beta})$
if either (i) $\nu >0$, or (ii) $\delta (y) = O (y^{-\delta})$
 for some $\delta >0$, and $-\kappa \delta < \nu <0$;
taking $\kappa <1$ close to $1$, we can permit any $\nu > -\delta$.
	Thus, in either case,
	from~\eqref{eq:B2} with~\eqref{eq:h-prime} we have that  $\Exp_x[B_{i,2}]$ is equal to $\nu x^{\nu - \beta}$ times
	\[  
-\int_0^{1-x^{-1}}  ((1-u)^{\nu-1} -1)  u^{-\beta} \ud u  +\int_{1-x^{-1}}^\infty   u^{-\beta}
\ud u +  (i-1) \int_{1+x^{-1}}^\infty  
(u-1)^{\nu-1}u^{-\beta} \ud u + o(1) . \]
A final application of Corollary~\ref{cor:int-negative-part-to-x-1} and
Lemma~\ref{lem:int-negative-part-from-x+1} (with $p=1-\beta$ and $q=\nu$)
yields expressions for $\Exp_x [ B_{i,2}]$ which, when combined with~\eqref{eqn:heavy-in-thetapos}
and the fact that $\Exp_x[ f_i (x- \theta_-) - f_i(x)] = \Exp_x[B_{i,2}] - \nu x^{\nu-1} \Exp_x[\theta_-] +
o(x^{\nu-\beta})$, give the claimed results.
\end{proof}

Finally, we state an analogous result in the case of balanced tails. 
The proof amounts to combining elements of the proofs of Lemmas~\ref{lem:lyapunov-R-out} and~\ref{lem:lyapunov-R-in}, and is omitted.

\begin{lemma}
\label{lem:lyapunov-R-bal}
Suppose that the random walk $\Xi$ on $\R$~satisfies~\eqref{ass:tails-on-R-bal}.
Then, for any $\nu$ for which $0 < \nu < \alpha$, the following hold.
(i) As $x \to +\infty$,
\[
D_1 (x)
= \nu x^{\nu-1} \Exp_x[\theta] +
c  x^{\nu -\alpha} (\nu \kappa_0 (\alpha,\nu) + \kappa_1 (\alpha,\nu))
+
o(x^{\nu-\alpha}).
\]
(ii) As $x \to \pm\infty$,
\begin{align*}
D_2 (x) &= \nu \sign(x) |x|^{\nu-1} \Exp_x[\theta] +
c \nu |x|^{\nu-\alpha} (\kappa_0 (\alpha,\nu) + \kappa_2 (\alpha,\nu) ) +
o(|x|^{\nu-\alpha}).
\end{align*}
If in addition,
the limit assumptions in~\eqref{ass:tails-on-R-bal}
are strengthened to $O(y^{-\delta})$ for some
 $\delta >0$, 
then~(ii) also holds for $-\delta_\star < \nu < 0$, where $\delta_\star = \min\{\delta, 1 \}$.
\end{lemma}

\section{Proofs of main results}
\label{sec:proofs}

\subsection{The half line}
\label{sec:half-line-proofs}

First we establish our recurrence criteria.

\begin{proof}[Proof of Theorem~\ref{thm:recurrence-on-RP}]
Under the conditions of part~(i) of the theorem, Lemma~\ref{lem:lyapunov-f0} gives
\begin{equation}
\label{eq:D0} D_0 (x) = \nu x^{\nu-\alpha} \left( x^{\alpha-1} \mu (x) + c \kappa_0 ( \nu) + o (1) \right), \text{ for } \nu \in (\alpha-\beta, \alpha), \end{equation}
where $\kappa_0$ is given at~\eqref{eq:kappa-0}.
If~\eqref{eq:drift-recurrence}
holds, then there are $\eps >0$ and $x_1 \in \RP$ such that
$x^{\alpha-1} \mu (x)  + c \pi \cosec(\pi\alpha) < -2\eps$ for all $x \geq x_1$.
Since $\kappa_0 (0) = \Gamma(\alpha) \Gamma (1-\alpha) = \pi\cosec(\pi\alpha)$,
we can find $\nu >0$ such that $c \kappa_0 (\nu) < c \pi\cosec(\pi\alpha) + \eps$.
Then $x^{\alpha-1} \mu (x) + c \kappa_0 ( \nu) \leq - \eps$ for all $x \geq x_1$,
so that the right-hand side of~\eqref{eq:D0} is negative for all $x$ sufficiently large.
Since $\nu >0$, $f_0 (x) \to \infty$ as $x \to \infty$, and Lemma~\ref{lem:rec-RP}
implies recurrence, noting~\eqref{ass:non-confinement-on-RP}.

Similarly, if~\eqref{eq:drift-transience} holds then we can find $\nu < 0$
such that the right-hand side of~\eqref{eq:D0} is again negative for all $x$ sufficiently large,
but now with $f_0(x) \to 0$ as $ x \to \infty$. Hence Lemma~\ref{lem:trans-RP}
implies transience, again noting~\eqref{ass:non-confinement-on-RP}.

Finally, consider part~(ii) of the theorem.
 Fix $\nu = \max \{ 1 + \gamma, 1\}$;
since $\gamma < \alpha -1$, we have $1 \leq \nu < \alpha$.
Then 
Lemma~\ref{lem:big-drift}
shows that, for any $\eps >0$,  
\[ D_0 (x) =  \nu x^{\nu-1-\gamma} \left( x^{\gamma} \mu (x) + O (x^{1+\gamma-\alpha+\eps} ) \right)  .\]
 Taking $\eps >0$ small enough and using the fact that $x^{\gamma} \mu (x) \leq - \delta$
for some $\delta >0$ and all $x$ sufficiently large, we get $D_0 (x) \leq - \delta \nu  + o (1)$,
and so Foster's criterion (see e.g.~Theorem~2.6.2 of~\cite{mpow}) yields positive recurrence.
\end{proof}

Before proving Theorem~\ref{thm:moments-on-RP} on moments of passage times,
we state a lemma that we will need for
our non-existence-of-moments results in the case where the drift
is dominant.

\begin{lemma}
\label{lem:big-jump}
Suppose that~\eqref{ass:tails-on-RP} holds, 
and 
$\liminf_{x \to \infty} ( x^\gamma \mu (x) ) > -\infty$ for some $\gamma$ with $0 \leq \gamma < \alpha -1$.
Then, for any $q \geq \frac{\alpha}{1+\gamma}$, all $a$ large enough,
and all $x >a$,
$\Exp_x [ \tau_a ^q] =\infty$.
\end{lemma}
\begin{proof}
Suppose that $x \geq x_0$.
The idea is that the chain
may start with a very big jump, and then takes a long time to return to near~$0$
(a similar idea was used in e.g.~the proof of Theorem~2.10 of~\cite{hmmw}).
Fix $\eta = 1+\gamma$, and for $x, y \in \RP$ let $w_y (x) = (y-x)^\eta \1 { x < y }$.
We claim that there exist $y_0 \in \RP$ and $B < \infty$ (not depending on $y$) such that
\begin{equation}
\label{eq:big-jump-claim} \Exp [ w_y ( \xi_{n+1} ) - w_y (\xi_n) \mid \xi_n = x] \leq B, \text{ for all } y \geq y_0 \text{ and all } x \geq y/2.\end{equation}
Given the claim, it follows from the maximal inequality in Theorem~2.4.7 of~\cite{mpow} that
\begin{equation}
\label{eq:big-jump-0}
\Pr \Big[ \min_{0 \leq m \leq n} \xi_m < y /2 \Bigmid \xi_0 > y \Big] \leq
 \Pr \Big[ \max_{0 \leq m \leq n} w_y (\xi_m ) \geq (y/2)^\eta \Bigmid \xi_0 > y \Big] \leq 4 B n y^{-\eta} ,\end{equation}
for all $y \geq y_0$ and all $n \geq 1$.
Let $a \geq x_0$. Then for $x > a$ and any $A \in (0,\infty)$,
\begin{equation}
\label{eq:big-jump-1}
 \Pr_x [ \tau_a > n ] \geq \Exp_x \left[ \1 { \xi_1 > A n^{1/\eta} } \Pr [ \tau_a > n \mid \xi_1 > A n^{1/\eta} ] \right] .\end{equation}
Choosing $y = A n^{1/\eta} > 2a$ in~\eqref{eq:big-jump-0} for $A$ sufficiently large, we get, for all $n \geq 1$,
\begin{equation}
\label{eq:big-jump-2}
  \Pr [ \tau_a > n \mid \xi_1 > A n^{1/\eta} ] \geq 1-  4 B A^{-\eta} \geq 1/2. \end{equation}
Combining~\eqref{eq:big-jump-1} and~\eqref{eq:big-jump-2} and using~\eqref{ass:tails-on-RP},
we get $\Pr_x [ \tau_a > n ] \geq c' n^{-\alpha/\eta}$
for some $c'>0$ and all $n$ large enough, uniformly in $x > a \geq x_0$. Since $\eta = 1+\gamma$, we get $\Exp_x [ \tau_a^q] = \infty$ for $x >a$ and $q \geq \frac{\alpha}{1+\gamma}$.

It remains to prove~\eqref{eq:big-jump-claim}. 
Under assumption~\eqref{ass:tails-on-RP},
we have $\sup_{x \geq x_0} \Exp_x [ |\theta|^q ] < \infty$
for any $q < \alpha$, so in particular this is true for $q = \eta$.
Since $x \mapsto w_y(x)$ is non-increasing,   if $x \geq y -1$ then
$w_y (x + \theta) \leq w_y (y - 1 - \theta_-) \leq ( 1 + \theta_- )^\eta$, so 
$\Exp_x [ w_y (x+\theta) - w_y (x) ] \leq B$ for any $x \geq y - 1$, if $y \geq y_0 \geq x_0 +1$, say.
Suppose now that $y/2 \leq x < y -1$. 
Then
\begin{align*}
 \Exp_x [ ( w_y(x+\theta) - w_y(x) )\1 { | \theta | > \tfrac{y-x}{2} } ] & \leq \Exp_x [ ( w_y(x-\theta_-) - w_y(x) )\1 { \theta_- > \tfrac{y-x}{2} } ]
\\
& \leq \Exp_x[(\theta_- + y-x )^\eta \1 { y-x < 2 \theta_-  }] \\
& \leq C \Exp_x[\theta_-^\eta] ,\end{align*}
which again is bounded, if $y \geq y_0 \geq 2x_0$, say.
By Taylor's theorem, for some $C < \infty$,
$| (1-z)^\eta - 1  + \eta z | \leq C z^2$
for all $|z| \leq 1/2$. 
Thus
\begin{align*}
\left|
 w_y ( x +\theta ) - w_y (x)   
+ \eta (y-x)^{\eta-1} \theta  
\right| \1 {   | \theta|  < \tfrac{y-x}{2}  }
\leq 
C \theta^2 (y-x)^{\eta-2}  \1 {   | \theta|  < \tfrac{y-x}{2}  }, \end{align*}
so that, since $\theta^2 = |\theta|^\eta |\theta|^{2-\eta} \leq C (y-x)^{2-\eta} | \theta |^\eta$ on the event $\{   | \theta|  < \tfrac{y-x}{2}   \}$, we have
\[ \Exp_x [ | w_y ( x +\theta ) - w_y (x) + \eta (y-x)^{\eta-1} \theta |  \1 {   | \theta|  <  \tfrac{y-x}{2}    } ]
= O( 1 ) .\]
Also $\Exp_x [ | \theta | \1 {   | \theta|  > \tfrac{y-x}{2}  } ] \leq C (y-x)^{1-\eta}$. Putting these pieces together we get
\[ \Exp_x [ w_y (x+\theta) - w_y (x) ] \leq C - \eta  (y-x)^{\eta-1} \Exp_x [ \theta  ] \leq B, \text{ for all } y/2 \leq x < y-1,\]
by choice of $\eta$ and the assumption on $\mu(x)$. This completes the proof of~\eqref{eq:big-jump-claim}.
\end{proof}

\begin{proof}[Proof of Theorem~\ref{thm:moments-on-RP}]
We first prove part~(ii) of the theorem.
Assuming that $x^{\alpha-1} \mu (x) = b+o(1)$, with $\kappa_0$ given by~\eqref{eq:kappa-0}
we get from~\eqref{eq:D0} that 
$D_0 (x) = \nu x^{\nu-\alpha} ( b + c \kappa_0 (\nu) + o(1) )$.
Recall that $\frac{\ud}{\ud z} \log \Gamma (z) = \psi (z)$ is
the digamma function, which is strictly increasing on $z >0$.
For $\nu < \alpha$,   we have
\[ \frac{\ud}{\ud \nu}  \kappa_0 (\nu ) = 
\frac{\Gamma (1-\alpha) \Gamma (\alpha-\nu)}{\Gamma (2-\nu)} \left( (1-\nu) \left(
\psi (2-\nu) - \psi(\alpha -\nu) \right) - 1\right)  .\]
For $\nu < \alpha \in (1,2)$ we have $\Gamma (1-\alpha) <0$, $\Gamma (\alpha -\nu) >0$,
and $\Gamma (2-\nu) >0$, and since $\psi ( \alpha - \nu ) < \psi (2 - \nu)$,
we have $\kappa'_0 (\nu ) > 0$ for $\nu \geq 1$.
For $\nu < 1$, using $\psi ( 2- \nu ) = \psi (1-\nu) + \frac{1}{1-\nu}$
and  $\psi(1-\nu) < \psi (\alpha -\nu)$,
we also get $\kappa_0' (\nu) >0$ for $\nu < 1$.
Thus $c\kappa_0$ is strictly increasing on $[0,\alpha)$,
with $c\kappa_0 (0) = c\pi\cosec(\pi\alpha) <0$
and $c\kappa_0 (\nu) \to +\infty$ as $\nu \to \alpha$.

Hence if $b + c\pi\cosec(\pi\alpha) <0$
there exists a unique $\nu^\star$ in $(0,\alpha)$
such that $c \kappa_0 (\nu^\star) = -b$, i.e.,
solving~\eqref{eqn:nu-star}, such that the following two statements hold.
\begin{enumerate}[(a)]
\item For all $\nu \in (0,\nu^\star)$
there exist $\delta >0$, $x_1 \in \RP$ so that
$D_0 (x) \leq -   \delta x^{\nu-\alpha}$ for all $x \geq x_1$.
\item For all $\nu \in (\nu^\star,\alpha)$
there exist $\delta >0$, $x_1 \in \RP$ so that
$D_0 (x) \geq \delta x^{\nu-\alpha}$ for all $x \geq x_1$.
\end{enumerate}
Since $f_0(x) = x^\nu$ for $x \geq 1$, we can rewrite the inequality in~(a) as
$D_0(x) \leq - \delta f(x)^{1-(\alpha/\nu)}$,
so~(a) implies that for any $\kappa < 1- (\alpha/\nu^\star)$ we can choose $\nu$ less than but arbitrarily close
to $\nu^\star$ so that~\eqref{eqn:finite-mom-super} holds for the process $Y_n =
f_0(\xi_n)$.  Therefore, for all $q < \nu^\star/\alpha$, Lemma~\ref{lem:finite-moments}
shows that $\Exp_x[\tau_a^q] < \infty$ for all $a \geq x_1$, the constant given in~(a), and any $x$.

On the other hand, since $D_0(x) =
O(x^{\nu-\alpha}) = o(1)$ for any $\nu \in (0,\alpha)$, the
process $Y_n = f_0(\xi_n)$ satisfies~\eqref{eqn:Y-inc-lower-bound}.
For $r \in
(1,\alpha/\nu)$, 
we note that
\begin{equation}
\label{eq:r}
   \Exp [ (f^\nu_0 (\xi_{n+1}))^r - (f^\nu_0 (\xi_n) )^r \mid \xi_n = x] 
=   D^{r\nu}_0 (x)   = O ( x^{r\nu-\alpha}) .\end{equation}
Since this is $O(f_0(x)^{r-(\alpha/\nu)})$, we see that~\eqref{eqn:Yr-inc-bound} also holds.  
Similarly, taking $r = p$ in~\eqref{eq:r} and using statement~(b),
we see that
inequality~\eqref{eqn:Yp-submart} holds for any $p \in (\nu^\star/\nu,\alpha/\nu)$.
Taking $\nu$ less than but arbitrarily close to $\alpha$ and applying
Lemma~\ref{lem:infinite-moments} we obtain, for all $q >
\nu^\star/\alpha$, all $a$ sufficiently large, and all $x >a$, that
$\Exp_x[\tau_a^q] = \infty$.

Next consider part~(i). In this case
$x^{\alpha-1} \mu(x) \to 0$,
so we may apply part~(ii) with $b=0$.
Clearly, the right hand side of~\eqref{eqn:nu-star} equals zero when $\nu^\star=1$, so this
is the (unique) solution for $b=0$, which yields the critical exponent $1/\alpha$, giving part~(i).

Finally, consider part~(iii).
Here,
we get from~\eqref{eq:D0} that 
$D_0 (x) = \nu x^{\nu-\gamma-1} ( b + o(1) )$,
where $x^{\nu-\gamma-1} = f_0(x)^{1-(\gamma+1)/\nu}$ for $x \geq 1$. 
So, for $b<0$ and any $\kappa <
1-\frac{\gamma+1}{\alpha}$, we can choose $\nu$ less than but arbitrarily close to $\alpha$ so
that~\eqref{eqn:finite-mom-super} holds for the process $Y_n = f_0(\xi_n)$.  Then
Lemma~\ref{lem:finite-moments} implies that $\Exp_x[\tau_a^q] <\infty$ for all $q < \frac{\alpha}{\gamma+1}$.
For the non-existence-of-moments result, we apply Lemma~\ref{lem:big-jump}
(note that we cannot use Lemma~\ref{lem:infinite-moments},
since here $Y_n$ is never a submartingale).
\end{proof}

\subsection{The whole line}
\label{sec:proofs-R}
 
First we deal with the case of heavier outwards tails. Recall the definition of $\kappa_0$ at~\eqref{eq:kappa-0}.

\begin{proof}[Proof of Theorem~\ref{thm:recurrence-on-R-out}]
Lemma~\ref{lem:lyapunov-R-out} with assumptions~\eqref{ass:tails-on-R-out} and~\eqref{ass:R-drift} shows that 
\begin{align}
\label{eq:D1}
 D_1 (x) & = \nu \left( (b + o(1) ) x^{\nu-1-\gamma} + ( c +o(1) ) x^{\nu -\alpha} \kappa_0 (\nu) \right), \text{ as } x \to +\infty, \\
\label{eq:D2}
 D_2 (x) & = \nu  \left( (b + o(1) ) |x|^{\nu-1-\gamma} + ( c +o(1) ) | x |^{\nu -\alpha} \kappa_0 (\nu) \right),  \text{ as } x \to \pm \infty. \end{align}
  Lemma~\ref{lem:lyapunov-R-out} and the symmetries of  assumptions~\eqref{ass:tails-on-R-out} and~\eqref{ass:R-drift} also gives
\begin{align}
\label{eq:D1-dash}
 D_1' (x) & := \Exp [ f_1 (-\xi_{n+1} ) - f_1 (-\xi_n) \mid -\xi_n = x] \\
& = \nu \left( (b + o(1) ) x^{\nu-1-\gamma} + ( c +o(1) ) x^{\nu -\alpha} \kappa_0 (\nu) \right), \text{ as } x \to +\infty .\nonumber
\end{align}

Suppose that $\gamma = \alpha +1$. Similarly to the proof of Theorem~\ref{thm:recurrence-on-RP},
  if $b + c \pi \cosec (\pi \alpha ) > 0$ then, since $\kappa_0 (0) = \pi\cosec(\pi\alpha)$,
 we can find $\nu < 0$
such that $\nu ( b + c \kappa_0 (\nu ) ) <0$, so that $D_1(x)<0$ and $D'_1(x) <0$ for all $x$ sufficiently large. Lemma~\ref{lem:directional-transience} with function $f_1$ then
implies directional transience, giving part~(v) of the theorem. If  $b + c \pi \cosec (\pi \alpha ) < 0$,
 then we can find $\nu > 0$
such that $\nu ( b + c \kappa_0 (\nu ) ) <0$ again, so, by~\eqref{eq:D2}, $D_2 (x) <0$ for all $|x|$ sufficiently large,
and then Lemma~\ref{lem:rec-R} with function $f_2$ implies recurrence.
The latter case also includes $b=0$, so we get the recurrence in both parts~(i) and~(iv) of the theorem.
The \emph{null} recurrence follows from Theorem~\ref{thm:moments-on-R-out}, which we establish below.

Finally, if $\gamma < \alpha -1$ then  
for all $x$ sufficiently large, $D_1(x)$ and 
$D'_1 (x)$ have the same sign as $\nu b$. For $b > 0$ we can choose $\nu <0$ such that
Lemma~\ref{lem:directional-transience} with function $f_1$ implies directional transience.
For $b <0$, we can take $\nu = 1+\gamma \in (0,\alpha)$ to see by~\eqref{eq:D2} that $D_2 (x) \leq - \eps$ for all $x$ outside
of a bounded set, so Foster's criterion (e.g.~Theorem~2.6.2 of~\cite{mpow}) gives positive recurrence.
This proves parts~(ii) and~(iii) of the theorem.
\end{proof}

\begin{proof}[Proof of Theorem~\ref{thm:moments-on-R-out}]
Suppose first that $\gamma = \alpha -1$. Then we have from~\eqref{eq:D2}
that
$D_2 (x) = \nu |x|^{\nu-\alpha} ( b + c \kappa_0 (\nu) + o(1))$.
This is the same asymptotics as for $D_0(x)$ in the proof of Theorem~\ref{thm:moments-on-RP},
and we can follow that proof with $Y_n = f_2 (\xi_n)$. Setting $b=0$ also gives the case $\gamma > \alpha -1$.
In the case $\gamma < \alpha -1$ and $b<0$, Lemma~\ref{lem:big-jump}
carries through for $x >0$, and a symmetric argument works for $x <0$.
\end{proof}

We turn to heavier inwards tails. Recall $\kappa_1$ and $\kappa_2$ from~\eqref{eq:kappa-1} and~\eqref{eq:kappa-2}.

\begin{proof}[Proof of Theorem~\ref{thm:recurrence-on-R-in}]
Now we have from Lemma~\ref{lem:lyapunov-R-in} that
\begin{align*}
 D_1 (x) & =  \left( \nu (b + o(1) ) x^{\nu-1-\gamma} + ( c +o(1) ) x^{\nu -\beta} \kappa_1 (\nu) \right), \text{ as } x \to +\infty, \\
 D_2 (x) & = \nu  \left( (b + o(1) ) |x|^{\nu-1-\gamma} + ( c +o(1) ) | x |^{\nu -\beta} \kappa_2 (\nu) \right),  \text{ as } x \to \pm \infty,\\
 D_1' (x) &  =  \left( \nu (b + o(1) ) x^{\nu-1-\gamma} + ( c +o(1) ) x^{\nu -\beta} \kappa_1 (\nu) \right), \text{ as } x \to +\infty .
\end{align*}
where $D_1'$ is as defined at~\eqref{eq:D1-dash}.
Suppose first that $\gamma = \beta -1$. Then
$D_2 (x) = \nu |x|^{\nu-\beta} ( b + c \kappa_2 (\nu) +o(1))$.
Using the fact that $\Gamma (\nu) \sim 1/\nu$ and 
$\Gamma (z +\nu)/\Gamma (z) = 1 + \nu \psi (z) + O (\nu^2)$
as $\nu \to 0$, 
we get
$\kappa_2 (\nu) =    \psi(1 -\beta) -\psi(\beta)   + O(\nu)$.
Here, $\psi(1-\beta)- \psi(\beta) = \pi \cot(\pi \beta)$.
So if $b + c \pi \cot (\pi \beta) < 0$, then we can find $\nu >0$
so that $D_2 (x) < 0$ for all $|x|$ sufficiently large, and Lemma~\ref{lem:rec-R}
implies recurrence. If $b + c \pi \cot (\pi \beta) > 0$, then we can find $\nu < 0$
so that again $D_2 (x) < 0$ for all $|x|$ sufficiently large, and Lemma~\ref{lem:trans-R}
implies transience (here we need~\eqref{ass:inwards-tail+} to allow us to take $\nu <0$
in Lemma~\ref{lem:lyapunov-R-in}). 
Moreover, $\kappa_1 (\nu) = -1 + O(\nu)$,
so that 
we can find $\nu > 0$
so that $D_1 (x) < 0$ and $D_1'(x) < 0$ for all $x$ sufficiently large,
and Lemma~\ref{lem:oscillatory-transience} shows that transience is oscillatory.
This proves parts~(iv) and~(vi) of the theorem, up to showing null recurrence
in part~(iv), which is covered by Theorem~\ref{thm:moments-on-R-in}.

Now suppose that $\gamma > \beta -1$. When $\beta < 3/2$, we can  simply
apply part~(vi) with $b=0$, and the fact that $\cot (\pi \beta) >0$ gives oscillatory transience.
For $\beta > 3/2$ we can apply part~(iv) to get null recurrence. Thus we verify parts~(i) and~(v) of the theorem.

Finally, suppose that $\gamma < \beta -1$.
Lemma~\ref{lem:big-drift}
shows that $D_1 (x) =  \nu b x^{\nu-\gamma-1} (1 +o(1))$,
for $\nu$ in an interval $(-\eps, \alpha)$,
and the same holds for $D_1'(x)$.
If $b<0$ we may take $\nu = 1 +\gamma < \alpha$ to get positive recurrence (e.g.~Theorem~2.6.2 of~\cite{mpow}),
establishing part~(ii).
If $b>0$ we may take $\nu <0$ to conclude directional transience
by an application of Lemma~\ref{lem:directional-transience} with function $f_1$, giving part~(iii).
\end{proof}

\begin{proof}[Proof of Theorem~\ref{thm:moments-on-R-in}]
Suppose first that $\gamma = \beta -1$. Then
$D_2 (x) = \nu |x|^{\nu-\beta} ( b + c \kappa_2 (\nu) + o(1))$.
The function $\kappa_2$ is continuous on $(-1,\beta)$
with 
$b + c\kappa_2 (0) =  b + c\pi \cot (\pi\beta) < 0$
and $\kappa_2 (\nu) \to +\infty$ as $\nu \to \beta$,
so the equation $b + c \kappa_2 (\nu) = 0$
has at least one solution in $(0,\beta)$.
Since $\kappa_2$ is analytic and non-constant on $[0,\beta)$,
the zeros of $b+c \kappa_2 (\nu)$ can accumulate only at $\nu =\beta$,
but this is ruled out since $\kappa_2 (\nu) \to +\infty$ as $\nu \to \beta$.
Thus there are finitely many $\nu \in (0,\beta)$ with $b + c \kappa_2 (\nu) = 0$.
Call the smallest one $\nu_\star$ and the largest $\nu^\star$.

To show that $\nu_\star = \nu^\star$, we adapt an idea from Lemma~11 in~\cite{mpew}.
Given $0 < \nu_1 \leq \nu_2 < \beta$,  by Jensen's inequality,
$\Exp_x [ (f_2^{\nu_1} (x + \theta))^{\nu_2/\nu_1} ] \geq \left( \Exp_x [ f_2^{\nu_1} (x +\theta ) ] \right)^{\nu_2/\nu_1}$, so that for $x \geq 1$,
\begin{align*}
f_2^{\nu_2} (x) + D_2^{\nu_2} (x) & \geq \left( f_2^{\nu_1} (x) + D_2^{\nu_1} (x) \right)^{\nu_2/\nu_1} \\
& = | x |^{\nu_2} \left( 1 + | x |^{-\nu_1} D_2^{\nu_1} (x) \right)^{\nu_2/\nu_1} \\
& = f_2^{\nu_2} (x) + \frac{\nu_2}{\nu_1} |x|^{\nu_2-\nu_1} D_2^{\nu_1} (x) + o ( |x|^{\nu_2 - \beta} ) ,
\end{align*}
using the fact that $D^\nu_2 (x) = \nu |x|^{\nu-\beta} ( b + c \kappa_2 (\nu) + o(1))$. Using this again and simplifying
we obtain $\kappa_2 (\nu_2) \geq \kappa_2 (\nu_1) + o(1)$, i.e., $\kappa_2 (\nu_2) \geq \kappa_2 (\nu_1)$.
Thus $\kappa_2$ is non-decreasing on $(0,\beta)$ and
so   $b + c\kappa_2 (\nu) = 0$ for all $\nu_\star \leq \nu \leq \nu^\star$. Thus we must have $\nu_\star=\nu^\star$.

Hence $b + c \kappa_2 (\nu) < 0$ for $\nu \in (0,\nu^\star)$
and $b + c \kappa_2 (\nu) >0$ for $\nu \in (\nu^\star,\beta)$.
The proof of part~(ii) of the theorem now follows that of Theorem~\ref{thm:moments-on-RP}(ii),
setting $Y_n = f_2 (\xi_n)$. 

If $\gamma > \beta -1$, we may apply the $b=0$ case of part~(ii),
assuming $\beta > 3/2$ so that $\cot (\pi\beta) < 0$.
We observe that $\kappa_2 (2\beta -3) = 0$,
so $\nu^\star = 2\beta -3$.
This gives part~(i).

Finally, if $\gamma < \beta -1$
 and $b<0$ we get from~\eqref{eq:D2} that
$D_2 (x) = \nu |x|^{\nu-\gamma-1} ( b + o(1))$ for $\nu \in (0,\beta)$,
and following the argument in the proof of Theorem~\ref{thm:moments-on-RP}(iii)
gives the existence-of-moments result. For the non-existence result,
 the argument for Lemma~\ref{lem:big-jump}
works with minor modifications: the big jump being taken in the other direction.
\end{proof}

Finally, we turn to the balanced case.

\begin{proof}[Proof of Theorem~\ref{thm:recurrence-on-R-bal}]
We give a sketch of the argument in the critical case where $\gamma = \alpha -1$.
For transience and recurrence, we use $f_2$ and get from Lemma~\ref{lem:lyapunov-R-bal}
that the sign of $D_2(x)$ for $\nu \approx 0$ and large $|x|$ is determined by
\[ b + c (\kappa_0 (0) + \kappa_2 (0) ) = b + c \pi ( \cosec (\pi \alpha ) + \cot (\pi \alpha ) ) = b + c \pi \cot \big( \tfrac{\pi \alpha}{2} \big) .
\]
In the critical case transience is again oscillatory, since in $D_1(x)$
the key quantity is $\nu (b + c\kappa_0 (0)) + \kappa_1 (0)$ which is $-1 + O(\nu)$.
We omit the details, as they are similar to the previous proofs.
\end{proof}

\begin{proof}[Proof of Theorem~\ref{thm:moments-on-R-bal}]
Omitted; again similar to previous proofs.
\end{proof}

\subsection{Example in the plane}

\begin{proof}[Proof of Theorem~\ref{thm:R2-example}]
For $\nu \in \R$ we define the Lyapunov function $g : \R^2 \to \RP$ by
\[ g (\bx) := \begin{cases} \| \bx \|^{\nu} & \text{for } \| \bx \| \geq 1, \\
1 & \text{for } \| \bx \| < 1 .
\end{cases} \]
Suppose throughout that $\| \bx \| \geq 1$.
Using assumption~\eqref{ass:R2-supp}, we can write 
\[ \Exp [ g (\xi_{n+1} ) - g (\xi_n) \mid \xi_n = \bx ] = p^\cR \Exp_\bx [ g (\bx + \bu_\bx \theta^\cR) - g(\bx) ]
+ p^\cT \Exp_\bx [ g (\bx + \bv_\bx \theta^\cT) - g(\bx) ] .\]
Considering first the case of a radial jump, since $\bx = \| \bx \| \bu_\bx$, we have
\[ \Exp_\bx [ g (\bx + \bu_\bx \theta^\cR) - g(\bx) ] = \Exp_\bx [ (\|\bx\| + \theta^\cR)^\nu - \| \bx \|^\nu ] .\]
In the same way that Lemma~\ref{lem:lyapunov-R-out} follows from
the assumptions~\eqref{ass:tails-on-R-out} and~\eqref{ass:R-drift} 
of Section~\ref{sec:R-results1}, here assumptions~\eqref{ass:R2-mart}, \eqref{ass:R2-rad-out}, and~\eqref{ass:R2-rad-in} imply that
for any $\nu \in (\alpha-\beta,\alpha)$,
\[
\Exp_\bx [ g (\bx + \bu_\bx \theta^\cR) - g(\bx) ] = c^\cR
\nu \|\bx \|^{\nu-\alpha} \kappa_0 (\alpha,\nu) + o(\| \bx \|^{\nu-\alpha}).
\]
(Note that the process $\| \xi_n \|$ does not itself satisfy assumptions~\eqref{ass:tails-on-R-out} and~\eqref{ass:R-drift} of
Section~\ref{sec:R-results1}, as the law of its increments is not sufficiently uniform in $\bx$.)

We  turn to the transverse jumps. Let $\eps \in (0,2-\alpha)$.
Since $\| \bx + \bv_\bx \theta^\cT \|^2 = \| \bx \|^2 + | \theta^\cT |^2$, 
\begin{align*}
\left| g(\bx + \bv_\bx \theta^\cT ) - g (\bx ) \right| \1 { | \theta^\cT |^2 \leq \| \bx \|^\eps }
& = \| \bx \|^\nu \left[ \left( 1 + \| \bx \|^{-2} | \theta^\cT |^2 \right)^{\nu/2} - 1 \right]  \1 { | \theta^\cT |^2 \leq \| \bx \|^\eps } ,
\end{align*}
which is bounded by a constant times $\| \bx \|^{\nu +\eps -2}$ for all $\| \bx \| \geq 1$.
Thus, by our choice of $\eps$,
\begin{equation}\label{eqn:trans-less-x-to-eps}
\Exp_\bx [ | g (\bx + \bv_\bx \theta^\cT ) - g(\bx) | \1 { | \theta^\cT |^2 \leq \| \bx \|^\eps } ] =
o(\|\bx\|^{\nu-\alpha}) .
\end{equation}
The function $g_\bx : \RP \to \R$ defined by $g_\bx(y) = (\|\bx\|^2 + y)^{\nu/2} - \|\bx\|^\nu$ is 
differentiable and monotone on $\RP$
with $g_\bx (y^2) = g( \bx + \bv_\bx y ) - g(\bx )$,  so Lemma~\ref{lem:exp-tail} implies that
\begin{align}
\label{eq:transverse-big}
& {} \Exp_\bx [ ( g (\bx  + \bv_\bx \theta^\cT ) - g(\bx) ) \1 { | \theta^\cT |^2 > \| \bx \|^\eps } ] \nonumber \\
& \qquad {} =g_\bx(\|\bx\|^\eps) \Pr_\bx [ | \theta^\cT |^2 > \|\bx\|^\eps] +
\int_{\|\bx\|^\eps}^\infty \frac{\nu}{2}(\|\bx\|^2 + y )^{(\nu/2)-1} \Pr_\bx[| \theta^\cT |^2> y]\ud y.
\end{align}
Similarly to~\eqref{eqn:trans-less-x-to-eps}, we have $g_\bx(\|\bx\|^\eps) = o(\|\bx\|^{\nu-\alpha})$. 
For the above integral, assumption~\eqref{ass:R2-trans} and the substitution $y =
\|\bx\|^2 u$ shows that~\eqref{eq:transverse-big} equals
\[
c^\cT \nu 
\|\bx\|^{\nu-\alpha}  \int_0^\infty (1+u)^{(\nu/2)-1} u^{-\alpha/2} \ud u +
o(\|\bx\|^{\nu-\alpha}),
\]
provided that $\alpha \in (1,2)$ and $\nu < \alpha$ so that the latter integral is convergent. Using the substitution $t = (1+u)^{-1}$, 
and combining the result with~\eqref{eqn:trans-less-x-to-eps}, gives
\[ \Exp_\bx [   g (\bx  + \bv_\bx \theta^\cT ) - g(\bx)    ]
= c^\cT  \nu \|\bx\|^{\nu-\alpha} ( 1 + o(1) ) \int_0^1 t^{\frac{\alpha-\nu}{2}-1}(1-t)^{-\alpha/2}\ud t .\]
Therefore, recalling the beta integral and the definition of $\kappa_0$ from~\eqref{eq:kappa-0}, we obtain
\[
 \Exp [ g (\xi_{n+1} ) - g (\xi_n) \mid \xi_n = \bx ] = \nu \|\bx \|^{\nu-\alpha}
\left(
p^\cR c^\cR
 \kappa_0 (\alpha,\nu) 
+  p^\cT c^\cT   \kappa_0 (\alpha/2,\nu/2)
+ o( 1) \right).
\]
We saw in the proof of Theorem~\ref{thm:recurrence-on-RP} that $\kappa_0 (\alpha,0) = \pi \cosec (\pi \alpha)$,
and using the fact that $\cosec (a) = 2 \cosec(2a) \cos(a)$,
 we get that the leading order coefficient   is
\[   \nu \pi \cosec( \pi \alpha) \left( p^\cR c^\cR  
+ 2 p^\cT c^\cT  \cos ( \tfrac{\pi\alpha}{2} ) 
  \right) + o(\nu) , \text{ as } \nu \to 0. \]
	
Since $\cosec(\pi \alpha) < 0$ for $\alpha \in (1,2)$, if $p^\cR c^\cR  + 2 p^\cT c^\cT
\cos(\pi\alpha/2) > 0$ then there exists a $\nu > 0$ for which
$\Exp [ g (\xi_{n+1} ) - g (\xi_n) \mid \xi_n = \bx ] \leq 0$ for all large enough $\|\bx\|$, and 
 we get the recurrence in part~(i) of the theorem by Lemma~\ref{lem:rec-RP} applied to $X_n = \| \xi_n \|$
with $f (x) = f_0 (x)$ (which tends to $\infty$).
On the other hand, if $p^\cR c^\cR + 2 p^\cT c^\cT \cos(\pi\alpha/2) < 0$, then there exists
a $\nu < 0$  for which again $\Exp [ g (\xi_{n+1} ) - g (\xi_n) \mid \xi_n = \bx ] \leq 0$  for all $\|\bx\|$
large enough,  
and we get the transience in part~(ii) of the theorem by Lemma~\ref{lem:trans-RP}
applied to  $X_n = \| \xi_n \|$
with $f (x) = f_0 (x)$ (which now tends to $0$). 

Finally, existence and non-existence of moments in the recurrent case follow by a 
similar argument to that in the proof of Theorem~\ref{thm:moments-on-RP} in
Section~\ref{sec:half-line-proofs}: since for any $\alpha \in (0,1) \cup (1,2)$ the function $\nu \mapsto \kappa_0 (\alpha,\nu)$ is
continuously increasing
on $[0,\alpha)$,
the expression
\[
p^\cR c^\cR
 \kappa_0 (\alpha,\nu) 
+  p^\cT c^\cT   \kappa_0 (\alpha/2,\nu/2)
\]
is finite and continuously increasing in $\nu$ on $[0,\alpha)$,
 by assumption it is negative for
$\nu = 0$, and it is clearly positive for $\nu = 1$. Hence there is a unique
$\nu^\star \in (0,1)$ that solves~\eqref{eqn:R2-nu-star}.  Applying
Lemmas~\ref{lem:finite-moments} and~\ref{lem:infinite-moments}, it is possible to show
that $\Exp_\bx[\tau_a^q] < \infty$ for $q < \nu^\star/\alpha$, and
$\Exp_\bx[\tau_a^q] = \infty$ for $q > \nu^\star/\alpha$.  In particular,
$\Exp_\bx[\tau_a] = \infty$ and $\Xi$ is null recurrent.
\end{proof}

\appendix

\section{An integration by parts formula}
\label{sec:parts}

The following lemma, 
based on the partial integration formula for the Riemann--Stieltjes integral, is 
a mild generalization of Theorem~2.12.3 of~\cite{gut}.

\begin{lemma}\label{lem:exp-tail}
Let $X \geq 0$ be a random variable, and $g : \RP \to \R$ a continuous function.
Let the finite partition $a_0 = 0 < a_1 < \cdots < a_{k+1} = \infty$ of
$\RP$ be such that $g$ is monotonic on $[a_i,a_{i+1})$ and differentiable on $(a_i,a_{i+1})$
for each $i=0,\dots,k$.  Then, for any $a \geq 0$,
\[
\Exp[g(X)\1{X>a}] = g(a)\Pr[X>a] + \int_a^\infty g'(x)\Pr[X> x]  \ud x,
\]
where both sides converge or diverge simultaneously.  
\end{lemma}
\begin{proof}
By assumption, for each $i=0,\dots,k-1$, the function $g$
is continuous on $[a_i,a_{i+1}]$  and $\Exp[g(X) \1{a_i < X \leq a_{i+1}}]$ is
finite and, using e.g.~Theorem~2.9.3 of~\cite{gut}, is equal to
\[\begin{split}
\int_{a_i}^{a_{i+1}} g(x) \ud F_X(x) &= g(a_{i+1})F_X(a_{i+1}) -g(a_i)F_X(a_i) -
\int_{a_i}^{a_{i+1}} F_X(x) g'(x) \ud x\\
&= g(a_i)\Pr[X > a_i] - g(a_{i+1})\Pr[X > a_{i+1}] + \int_{a_i}^{a_{i+1}} g'(x) \Pr[X>x]
\ud x.
\end{split}
\]
If $a< a_k$ then $a \in [a_j,a_{j+1})$ for some $j<k$, and similar reasoning gives
\[
\Exp[g(X) \1{a < X \leq a_{j+1}}] = g(a)\Pr[X > a] - g(a_{j+1})\Pr[X > a_{j+1}] +
\int_{a}^{a_{j+1}} g'(x) \Pr[X>x] \ud x.
\]
Consequently, the result follows if
\[
\Exp[g(X)\1{X>a_k}] = g(a_k)\Pr[X>a_k] + \int_{a_k}^\infty g'(x)\Pr[X> x]  \ud x
\]
and both sides converge or diverge simultaneously.  In other words, without loss of
generality, we may assume that $a \geq a_k$.

Under this assumption, we may also assume that $g$ is non-negative and non-decreasing on
$[a,\infty)$, by passing to $\pm(g(x) - g(a))$ as necessary, and therefore for all $b>
a$,
\[
 \Exp[g(X) \1{X>a}] \geq \Exp[g(X) \1{a<X\leq b}] + g(b) \Pr[X>b] \geq  \Exp[g(X) \1{a<X\leq b}].
\]
Applying the partial integration formula to the middle term above, we obtain
\[
 \Exp[g(X) \1{X>a}] \geq \int_a^b g'(x) \Pr[X>x]\ud x +  g(a)\Pr[X>a] \geq  \Exp[g(X) \1{a<X\leq b}].
\]
and the result follows by taking the limit $b \to \infty$, since the right hand side tends
to the (possibly infinite) limit $\Exp[g(X)\1{X>a}]$ by monotone convergence.
\end{proof}

\section{Some useful integrals}
\label{sec:integrals}

Recall that the beta function
$B (p,q)$ is given by
\[ B(p,q ) = \frac{\Gamma(p)\Gamma(q)}{\Gamma(p+q)} = B(q,p) .\]
The \emph{incomplete beta function} $B_x(p,q)$, defined for $0 \leq x \leq 1$ by the integral
\begin{equation}
\label{eq:beta-integral}
B_x(p,q) := \int_0^x u^{p-1} (1-u)^{q-1} \ud u,
\end{equation}
is usually only defined for $p,q > 0$, in which case $B_1 (p,q)= B(p,q)$.
However, the integral in~\eqref{eq:beta-integral} is finite for \emph{any} $q \in \R$ provided that
$x <1$ and $p> 0$, so we extend the definition of $B_x(p,q)$ to this full range of the parameters.
Note that then
\begin{equation}\label{eqn:beta-limit}
\lim_{x\uparrow 1} B_x(p,q) = \begin{cases}
B(p,q) & \text{if $q>0$,}\\
\infty & \text{if $q\leq 0$.}
\end{cases}
\end{equation}
The following recurrence relation
is well known, but since it is usually assumed that $q>0$, we repeat the proof here.
\begin{lemma}\label{lem:B-recur}
For $0\leq x< 1$, $p > 0$, and $q\in \R$,
\[
q B_x(p,q) = (p+q) B_x(p,q+1) - x^p(1-x)^q.
\]
\end{lemma}
\begin{proof}
Since $p>0$, we can write $B_x(p,q) = B_x(p,q+1) + B_x(p+1,q)$, where all three terms are
finite.  Evaluating $q B_x(p+1,q)$ using integration by parts, yields
\[
q B_x(p+1,q) = -x^p(1-x)^q  + p B_x(p,q+1),
\]
which when combined with the previous identity gives the result.
\end{proof}

\begin{lemma}\label{lem:beta-const}
For $0 \leq x < 1$, $p>-1$, and $q\in \R$,
\[
p \int_0^x u^{p-1}( (1-u)^{q-1} - 1)\ud u = (p+q) B_x(p+1,q) + x^p((1-x)^q-1).
\]
Moreover, for $p> -1$, $p\neq 0$, and $q > 0$,
\begin{equation}\label{eqn:beta-const}
\lim_{x\uparrow 1} \int_0^x u^{p-1}( (1-u)^{q-1} - 1)\ud u = (p+q)\frac{\Gamma(p)\Gamma(q)}{\Gamma(p+q+1)} - \frac{1}{p}.
\end{equation}
\end{lemma}

\begin{remark}
The identity~\eqref{eqn:beta-const} is easily proved when $p>0$ and $q>0$
(direct evaluation with the beta integral) or when $p>-1$, $p\neq 0$ and $q>1$ (integration by parts); however, the
region $-1<p<0$ and $0<q<1$ is not covered by either of these cases.
\end{remark}
\begin{proof}
Let $\phi (u) = u^{p-1}( (1-u)^{q-1} - 1)$.
Taylor's theorem around $u =0$ gives an $\eps \in (0,x)$
such that $| \phi (u) |$ is bounded by a constant times $u^p$
for $u \in [0,\eps]$, while, as $x<1$,
 $| \phi (u) |$ is uniformly bounded for $u \in [\eps, x]$.
Thus, since $p>-1$, the integral is finite.
Integrating by parts, noting that $\lim_{ u \to 0} u^p ( (1-u)^{q-1} -1 ) = 0$ since $p>-1$,
we get
\[
p\int_0^x u^{p-1}( (1-u)^{q-1} - 1)\ud u = x^p((1-x)^{q-1} - 1) +
(q-1) B_x(p+1,q-1),
\]
and applying Lemma~\ref{lem:B-recur} yields the claimed identity.
When $q>0$, by the first part of the lemma and~\eqref{eqn:beta-limit},
the integral $\int_0^x \phi (u) \ud u$ converges as
$x\to 1$ to a finite limit, which equals $\frac{p+q}p
B(p+1,q) - 1/p$, and gives the claimed expression since $p\Gamma(p) = \Gamma(p+1)$.
\end{proof}

\begin{lemma}\label{lem:beta-linear}
For  $0 \leq x < 1$, $p>-1$, and $q\in \R$,
\begin{align*}
& {} p(p-1) \int_0^x u^{p-2}((1-u)^q+qu-1)\ud u = (p+q)(p+q+1)B_x(p+1,q+1)\\
& {}
 \qquad \qquad\qquad \qquad\qquad \qquad {}
+ x^{p-1}((1-x)^{q+1}(p+(p+q)x) +(p-1)qx-p).
\end{align*}
Moreover, for $p>-1$, $p\notin \{0,1\}$, and $q > -1$,
\begin{equation}\label{eqn:beta-linear}
\lim_{x\uparrow 1} \int_0^x u^{p-2}((1-u)^q+qu-1)\ud u = (p+q)(p+q+1)\frac{\Gamma(p-1)\Gamma(q+1)}{\Gamma(p+q+2)} +
 \frac{q}{p} - \frac{1}{p-1}.
\end{equation}
\end{lemma}
\begin{proof}
The indefinite integral is finite for $x<1$ and $p>-1$, by a similar argument
to that in the proof of the Lemma~\ref{lem:beta-const}, since the integrand behaves like
$u^p$ as $u\to 0$.  Then, integrating by parts yields
\[
(p-1)\int_0^x u^{p-2}((1-u)^q+qu-1)\ud u = x^{p-1}((1-x)^q+qx-1) + q  \int_0^x u^{p-1}(
(1-u)^{q-1} - 1)\ud u, 
\]
and the required identity follows after applying Lemma~\ref{lem:beta-const} and then Lemma~\ref{lem:B-recur}.
When $q>-1$, using the first part of the lemma  and~\eqref{eqn:beta-limit},
the integral $\int_0^x u^{p-2}((1-u)^q+qu-1)\ud u$ converges as $x\to 1$ and, provided $p \notin \{0,1\}$,
this limit equals
\[
\frac{(p+q)(p+q+1)}{p(p-1)} B(p+1,q+1) + \frac{q}{p} - \frac{1}{p-1},
\]
and the result follows since $p(p-1) \Gamma(p-1) = \Gamma(p+1)$.
\end{proof}

We can now collect the statements that we will need in the body of the paper. 
The first, Corollary~\ref{cor:int-negative-part-to-x-1}, is
for  $q>0$ no stronger than the identity~\eqref{eqn:beta-const} above, but for $q \in (-1,0)$ 
it
quantifies the rate of divergence of the integral as $x\to \infty$.

\begin{corollary}\label{cor:int-negative-part-to-x-1}
Let $p>-1$, $p \neq 0$ and $q>-1, q\neq 0$. As $x \to \infty$,
\[
\int_0^{1-x^{-1}} u^{p-1} ((1-u)^{q-1}-1) \ud u =
-\dfrac{x^{-q}}{q} + (p+q)(p+q+1)\frac{\Gamma(p)\Gamma(q)}{\Gamma(p+q+2)} - \frac{1}{p} + o(1).
\]
\end{corollary}
\begin{proof}
The case $p=1$ follows by a direct evaluation of the integral.
Otherwise, for any $q \in \R$, and $p>-1$, $p\notin \{ 0,1 \}$,
integration by parts yields
\[\begin{split}
-q \int_0^{1-x^{-1}} u^{p-1} ((1-u)^{q-1}-1) \ud u &= (1-x^{-1})^{p-1}(x^{-q} +
 q(1-x^{-1}) -1)\\
&\qquad - (p-1) \int_0^{1-x^{-1}} u^{p-2} ((1-u)^q+qu-1) \ud u,
\end{split}
\]
which is valid for any $x>0$, since $p>-1$.   
For $q> -1$,
\[(1-x^{-1})^{p-1}(x^{-q} +  q(1-x^{-1}) -1) = x^{-q} +q-1 + o(1),
\]
and the claimed identity follows from~\eqref{eqn:beta-linear}.
\end{proof}

The remaining two results are more straightforward.

\begin{lemma}\label{lem:int-negative-part-from-x+1}
For $p,q \in \R$ with $q>-1, q\neq 0$, and $p+q < 1$, as $x\to \infty$,
\[
\int_{1+x^{-1}}^\infty u^{p-1}(u-1)^{q-1} \ud u =
-\frac{x^{-q}}{q} + (1-p)\frac{\Gamma(1-p-q)\Gamma(q)}{\Gamma(2-p)} + o(1).
\]
\end{lemma}
\begin{proof}
Integrating by parts, we find
\[
q \int_{1+x^{-1}}^\infty u^{p-1}(u-1)^{q-1} \ud u  = -(1+x^{-1})^{p-1} x^{-q} - (p-1) \int_{1+x^{-1}}^\infty
u^{p-2}(u-1)^q \ud u,
\]
which is valid for any $x>0$, since $p+q<1$.  The change of variable $v =
u^{-1}$ yields 
\[
\int_{1+x^{-1}}^\infty
u^{p-2}(u-1)^q \ud u = \int_0^{\frac{x}{x+1}} v^{-p-q}(1-v)^q \ud v = B(1-p-q,q+1) + o(1),
\]
by~\eqref{eqn:beta-limit}, provided $q>-1$.
Now note that $-(1+x^{-1})^{p-1}x^{-q} = -x^{-q} + o(1)$ for $q>-1$.
\end{proof}

\begin{lemma}\label{lem:int-positive-part}
For $p,q \in \R$ with $-1 < p <0$ and $p+q<1$,
\[
\int_0^\infty u^{p-1}((1+u)^{q-1}-1) \ud u= (1-q)\frac{\Gamma(1-p-q)\Gamma(p)}{\Gamma(2-q)}.
\]
\end{lemma}
\begin{proof}
Integrating by parts, we find
\[
p \int_0^\infty u^{p-1}((1+u)^{q-1}-1)\ud u = -(q-1) \int_0^\infty
u^p(1+u)^{q-2} \ud u
\]
which is valid since $-1 < p <0$ and $p+q<1$.  Setting $v= (1+u)^{-1}$ yields
\[
\int_0^\infty u^p(1+u)^{q-2} \ud u = \int_0^1 v^{-p-q} (1-v)^p \ud v = B(1-p-q,p+1),
\]
and the identity follows using the fact that $p\Gamma(p) = \Gamma(p+1)$.
\end{proof}

\section*{Acknowledgement}

NG was supported by the Heilbronn Institute for Mathematical Research.

\end{document}